\documentclass[11pt]{article}
\usepackage{dsfont,amsmath,amsfonts,amssymb,amsthm,mathrsfs}
\usepackage{fullpage}
\usepackage{graphicx,graphics}
\usepackage{epsfig}
\usepackage{latexsym}
\usepackage[applemac,utf8]{inputenc}
\usepackage{ae,aecompl}
\usepackage[english]{babel}
\usepackage[colorlinks=true]{hyperref}
\usepackage[toc,page]{appendix}
\usepackage{bbm}
\usepackage[T1]{fontenc}
\usepackage[utf8]{inputenc}

\date{}
\newcommand{\E}{\mathbb{E}}
\newcommand{\e}{{\mathrm e}}
\newcommand{\R}{{\mathbb{R}}}
\newcommand{\T}{\mathcal{T}}

\newcommand{\N}{\mathbb{N}}
\newcommand{\D}{\mathcal{D}}
\newcommand{\Z}{\mathbb{Z}}

\newcommand{\veps}{\varepsilon}

\newcommand{\pr}{\mathbb{P}}
\newcommand\numberthis{\addtocounter{equation}{1}\tag{\theequation}}
\renewcommand{\geq}{\geqslant}
\renewcommand{\leq}{\leqslant}
\renewcommand{\le}{\leqslant}
\newcommand{\G}{\vec{\mathcal{G}}}
\def\llbracket{[\hspace{-.10em} [ }
\def\rrbracket{ ] \hspace{-.10em}]}

\newcommand{\cv}{\ensuremath{\overset{\mathrm{(d)}}\longrightarrow}}

\newcommand{\exc}{\mathbf{e}}

\renewcommand{\theequation}{\arabic{equation}}

\newtheorem{thm}{Theorem}[section]

\newtheorem{lem}[thm]{Lemma}
\newtheorem{prop}[thm]{Proposition}

\newtheorem{rem}[thm]{Remark}

\date{}

\title{\bf The scaling limit of a critical random directed graph}

\author{Christina Goldschmidt\thanks{Department of Statistics and Lady Margaret Hall, University of Oxford, goldschm@stats.ox.ac.uk}\quad \& \hspace{0.2cm}Robin Stephenson\thanks{School of Mathematics and Statistics, University of Sheffield, robin.stephenson@normalesup.org }}

\begin{document}
\maketitle

\begin{abstract}
We consider the random directed graph $\vec{G}(n,p)$ with vertex set $\{1,2,\ldots,n\}$ in which each of the $n(n-1)$ possible directed edges is present independently with probability $p$. We are interested in the strongly connected components of this directed graph.  A phase transition for the emergence of a giant strongly connected component is known to occur at $p = 1/n$, with critical window $p= 1/n + \lambda n^{-4/3}$ for $\lambda \in \R$. We show that, within this critical window, the strongly connected components of $\vec{G}(n,p)$, ranked in decreasing order of size and rescaled by $n^{-1/3}$, converge in distribution to a sequence $(\mathcal{C}_1,\mathcal{C}_2,\ldots)$ of finite strongly connected directed multigraphs with edge lengths which are either 3-regular or loops.  The convergence occurs in the sense of an $\ell^1$ sequence metric for which two directed multigraphs are close if there are compatible isomorphisms between their vertex and edge sets which roughly preserve the edge lengths.  Our proofs rely on a depth-first exploration of the graph which enables us to relate the strongly connected components to a particular spanning forest of the undirected Erd\H{o}s--R\'enyi random graph $G(n,p)$, whose scaling limit is well understood.    We show that the limiting sequence $(\mathcal{C}_1,\mathcal{C}_2,\ldots)$ contains only finitely many components which are not loops. If we ignore the edge lengths, any fixed finite sequence of 3-regular strongly connected directed multigraphs occurs with positive probability.
\end{abstract}

\section{Introduction and main result}
Many real-world networks are inherently directed in nature.  Consider, for example, the World Wide Web: hyperlinks point from one webpage to another but the link in the other direction is not necessarily present.  However, this structurally important feature is often ignored in modelling, and the corresponding mathematical literature is much less well-developed.  In this paper, we consider the simplest possible model of a random directed graph and endeavour to understand the way in which its directed connectivity properties change as we adjust its parameters.

Let $\vec{G}(n,p)$ be a random directed graph with vertex set $[n]:= \{1,\ldots,n\}$ and random edge set where each of the $n(n-1)$ possible edges $(i,j)$, $i \neq j$, is present independently with probability $p$. We are interested in the \emph{strongly connected components} of $\vec{G}(n,p),$ that is the maximal subgraphs for which there exists a directed path from a vertex to any other. 

The usual Erd\H{o}s--R\'enyi random graph, $G(n,p)$, in which each of the $n(n-1)/2$ possible \emph{undirected} edges is present independently with probability $p$, will play an important role in our results.  It is well known that $G(n,p)$ undergoes a phase transition \cite{ErdosRenyi}: if $np \to c > 1$ as $n \to \infty$ then $G(n,p)$ has a unique giant component with high probability, while if $np \to c < 1$ as $n \to \infty$ then the components of $G(n,p)$ are of size $O_{\mathbb{P}}(\log n)$.  In the so-called \emph{critical window}, where $p = \frac{1}{n} + \lambda n^{-4/3}$, Aldous~\cite{aldous1997} proved that the sequence of sizes of the largest components possesses a distributional limit when renormalised by $n^{2/3}$.

Previous work by Karp~\cite{karp1990} and \L uczak~\cite{L1990} has shown that $\vec{G}(n,p)$ undergoes a similar phase transition to that of $G(n,p)$: if $np \to c>1$ as $n \to \infty$, then $\vec{G}(n,p)$ has a unique giant strongly connected component with high probability, while if $np\to c<1$ as $n \to \infty$, then the sizes of all the strongly connected components are $o_{\mathbb{P}}(n)$.  (In $\vec{G}(n,p)$, in contrast to the situation for $G(n,p)$, \L uczak~\cite{L1990} shows that the sizes of the subcritical components are tight.)  These results were strengthened by \L uczak and Seierstad~\cite{LS}, who showed that $\vec{G}(n,p)$ has, in fact, the same critical window as $G(n,p)$.

\begin{thm}[\L uczak and Seierstad~\cite{LS}]
Let $\gamma_n = (np-1)n^{1/3}$ and assume $\gamma_n=o(n^{1/3})$ as $n\to\infty.$
\begin{itemize}
\item[(i)] If $\gamma_n \to \infty$ then the largest strongly connected component of $\vec{G}(n,p)$ has size $(4+o_{\mathbb{P}}(1)) \gamma_n^2 n^{1/3}$ and the second largest has size $O_{\mathbb{P}}(\gamma_n^{-1} n^{1/3})$.
\item[(ii)] If $\gamma_n \to -\infty$ then the largest strongly connected component of $\vec{G}(n,p)$ has size $O_{\mathbb{P}}(|\gamma_n^{-1}| n^{1/3})$.
\end{itemize}
\end{thm}

However, in contrast to $G(n,p)$, \L uczak and Seierstad also show that within the critical window, the complex strongly connected components (that is, those which do not just consist of a single directed cycle) occupy only $O_{\mathbb{P}}(n^{1/3})$ vertices in total.  This shows that the critical components are very much ``thinner'' objects than in the setting of $G(n,p)$, where the complex components occupy $O_{\mathbb{P}}(n^{2/3})$ vertices.  

In a recent preprint \cite{Coulson}, Coulson shows that, on rescaling by $n^{-1/3}$, the size of the largest strongly connected component of $\vec{G}(n,p)$ in the critical window is tight, with explicit upper and lower tail bounds.

\begin{figure}[h]
\centering
\includegraphics[scale=0.7]{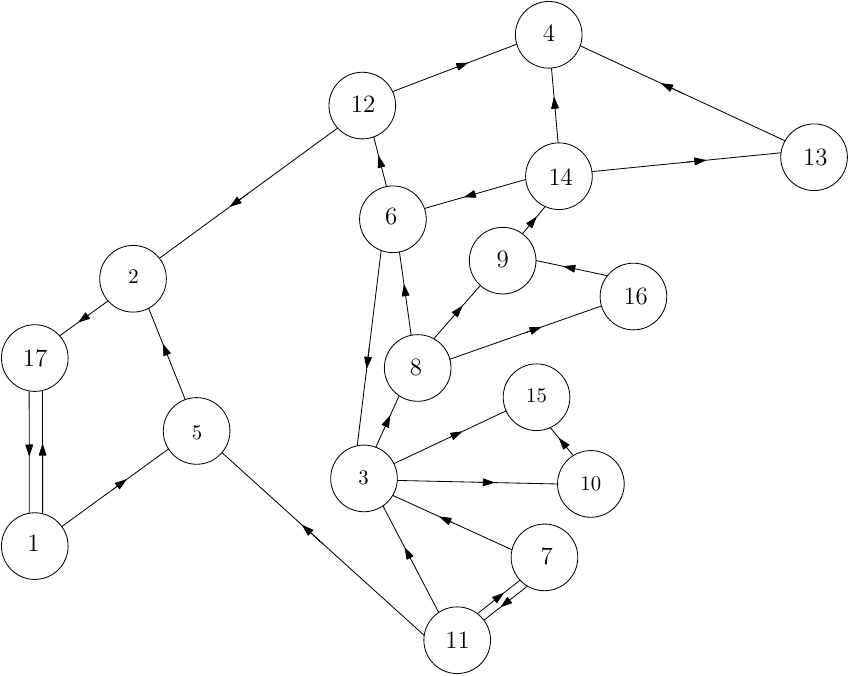}
\caption{A directed graph on $[17].$ Its strongly connected components have vertex sets $\{3,6,8,9,14,16\},\{1,2,5,17\},\{7,11\},\{4\},\{10\},\{12\},\{13\}$ and $\{15\}.$}
\label{fig1}
\end{figure}  

In this paper, we investigate the behaviour within the critical window in more detail, and in particular we prove a scaling limit for the strongly connected components.  We do this by relating a particular subgraph of $\vec{G}(n,p)$ to a spanning forest of $G(n,p)$, and the convergence of that spanning forest (thought of as a collection of discrete metric spaces, one per component) to a collection of random $\R$-trees. Similar tools have already been used to study the components of $G(n,p)$ in the same critical window, leading to the main theorem of \cite{A-BBG12}. 

\begin{thm}[Addario-Berry, Broutin and Goldschmidt~\cite{A-BBG12}] \label{thm:Gnp}
Let $p = p(n) = \frac{1}{n} + \lambda n^{-4/3}$ for fixed $\lambda \in \R$.  Let $(A_1(n),A_2(n),\ldots)$ be the connected components of $G(n,p)$, each considered as a metric space by endowing the vertex-set with the graph distance.  Then
\[
\left(\frac{A_i(n)}{n^{1/3}},i\in\N \right) \cv  (\mathcal{A}_i,i\in\N),
\]
where $\mathcal{A}=(\mathcal{A}_i,i\in\N)$ is a random sequence of compact metric spaces, and the convergence is in distribution for the $\ell^4$ metric for sequences of compact metric spaces based on the Gromov--Hausdorff distance. 
\end{thm}

Let us immediately give a description of the scaling limit $\mathcal{A}$, since it plays an important role in the sequel. Define $W^{\lambda}(t)=W(t)+\lambda t -t^2/2$ for $t \ge 0$, where $W$ is a standard Brownian motion, and let $(\sigma_i,i\in\N)$ be the collection of excursion lengths above the running infimum of $W^{\lambda}$, ranked in decreasing order. For $\sigma>0,$ let $\tilde{\exc}^{(\sigma)}$ be a Brownian excursion with length $\sigma$ biased by the exponential of its area, so that if $\exc^{(\sigma)}$ is a Brownian excursion of length $\sigma$ then, for any non-negative measurable test function $g$, we have
\[
\E \left[g \left(\tilde{\exc}^{(\sigma)} \right) \right] = \frac{ \E \left[\exp \left( \int_0^{\sigma} \exc^{(\sigma)}(u) \mathrm{d} u \right) g \left(\exc^{(\sigma)} \right) \right] }{\E \left[\exp \left( \int_0^{\sigma} \exc^{(\sigma)}(u) \mathrm{d} u \right) \right] }.
\]
Let $\T_{\sigma}$ be the $\R$-tree encoded by $2\tilde{\exc}^{(\sigma)}$ (see Section~\ref{subsec:rtrees} below for a description of how this is done).  We make some additional point-identifications in this tree. Let $(t_1,\ldots,t_K)$ be the points of a Poisson random measure on $[0,\sigma]$ with intensity $\tilde{\exc}^{(\sigma)}(t)\mathrm dt.$ The point $t_j \in [0,\sigma]$ corresponds to a point $x_j$ in $\T_{\sigma}$ at distance $2 \tilde{\exc}^{(\sigma)}(t_j)$ from the root.  For all $1 \le j \le K$, we identify $x_j$ with a uniformly chosen point on its path to the root.  Write $\mathcal{G}_{\sigma}$ for the resulting metric space.  Finally, conditionally on $(\sigma_i,i\in\N),$ the metric spaces $\mathcal{A}_1, \mathcal{A}_2, \ldots$ are independent and, for each $i\in\N,$ $\mathcal{A}_i$ has the law of $\mathcal{G}_{\sigma_i}$.

\bigskip

While metric spaces provide the natural setting in which to consider scaling limits of undirected graphs, this is no longer the case in the directed setting: we need some extra structure to encode the orientations.  Let us make some useful definitions.

By a \emph{directed multigraph}, we mean a triple $(V,E,r)$ where
\begin{itemize}
\item  $V$ and $E$ are finite sets.
\item  $r=(r_1,r_2)$ is a function from $E$ to $V\times V$, with $r_1(e)$ and $r_2(e)$ for $e\in E$ being respectively the \emph{tail} and \emph{head} of the directed edge $e$.
\end{itemize}
We will refer to the case where $V = \{v\}$, $E = \{e\}$ and $r_1(e) = r_2(e) = v$ as a \emph{loop}.  $X=(V,E,r,\ell)$ is a \emph{metric directed multigraph} (henceforth MDM) if $(V,E,r)$ is a directed multigraph and $\ell$ is a function from $E$ to $[0,\infty)$ which assigns each edge a length.  A special role will be played by the degenerate case of a loop whose single edge is assigned length 0, which we denote by $\mathfrak{L}$. The \emph{length} $\mathrm{len}(X)$ of $X$ is given by $\sum_{e \in E} \ell(e)$.

We now define a distance between MDMs $X=(V,E,r,\ell)$ and $X'=(V',E',r',\ell')$ in such a way that they are close if there is a graph isomorphism from $X$ to $X'$ which changes the lengths very little. Specifically, let $\mathrm{Isom}(X,X')$ be the set of graph isomorphisms from $X$ to $X',$ that is pairs of bijections $f$ from $V$ to $V'$ and $g$ from $E$ to $E'$ such that, for all $e\in E$, $r'(g(e))=(f(r_1(e)),f(r_2(e))).$ Then set
\[
d_{\G}(X,X')=\inf_{(f,g)\in\mathrm{Isom}(X,X')} \ \sup_{e\in E} \ |\ell(e)-\ell'(g(e))|.
\]
Note that if $X$ and $X'$ do not have the same graph structure, then $\mathrm{Isom}(X,X')$ is empty and $d_{\G}(X,X')$ is set to infinity.  Let $\G$ be the set of (isometry classes of) MDMs.  Then $(\G, d_{\G})$ is a Polish space (as a countable union of powers of $[0,+\infty)$).

Let $C_i(n)$ for $i \ge 1$ be the strongly connected components of $\vec{G}(n,p)$, listed in decreasing order of size, breaking ties by increasing order of the lowest labelled vertex. We view these strongly connected components as MDMs, by assigning to each edge a length of $1$, and then removing all vertices with degree $2$ and merging their corresponding edges into paths of length greater than $1$. In the case of a strongly connected component which consists of a single directed cycle with $k\ge 2$ vertices, we think of it as a loop of length $k$.  Similarly, we think of isolated vertices as loops of length $0$. Finally, since there are at most $n$ components, we complete the list with an infinite repeat of $\mathfrak{L},$ the loop of length $0$.

We can now state our main theorem.
\begin{thm}\label{thm:main} Suppose $p = p(n) = \frac{1}{n} + \lambda n^{-4/3} + o(n^{-4/3})$.  There exists a sequence $\mathcal{C}=(\mathcal{C}_i,i\in\N)$ of random strongly connected MDMs such that, for each $i \ge 1$, $\mathcal{C}_i$ is either 3-regular or a loop, and such that
\begin{equation}\label{eq:mainconvergence}
\left(\frac{C_i(n)}{n^{1/3}},i\in\N\right) \cv  (\mathcal{C}_i,i\in\N)
\end{equation}
with respect to the distance $d$ defined by
\[
d(\mathbf{A},\mathbf{B})=\sum_{i=1}^{\infty} d_{\G}(A_i,B_i),
\]
for $\mathbf{A}, \mathbf{B} \in \G^{\N}$.
\end{thm}

In particular, the limit object $\mathcal{C}$ has finite total length.  We will show later that $\mathcal{C}$ has only finitely many complex components (i.e.\ components which are not loops).  So Theorem~\ref{thm:main} implies the convergence in distribution of the number of complex components of $\vec{G}(n,p),$ their rescaled numbers of vertices, and their excesses (where the excess of a component is given by its number of edges minus its number of vertices). This, in particular, significantly strengthens Theorems 13 and 14 of \cite{LS}.  Finally, we also show that, if we ignore the edge lengths, then any fixed finite sequence of 3-regular strongly connected directed multigraphs occurs with positive probability.

\bigskip

We defer a proper description of $\mathcal{C}$, which is rather involved, to Section~\ref{sec:scalinglimit} below.  As is the case for $(\mathcal{A}_i,i\in\N),$ the $(\mathcal{C}_i,i\in\N)$ are derived from the $\R$-trees encoded by the excursions of $W^{\lambda}$.  However, the strongly connected components $(\mathcal{C}_i,i\in\N)$ are much simpler objects than $(\mathcal{A}_i, i \in \N)$ which, for example, have a rich fractal structure coming from their relationship to the Brownian continuum random tree.  A closer analogy is obtained by instead looking at the scaling limit of the subgraph of $G(n,p)$ consisting only of edges and vertices which lie in cycles.  Each component of the graph has a \emph{core}, which is defined to be the maximal subgraph of minimum degree 2, and consists of the vertices and edges which lie in cycles, as well as those in paths joining cycles.  (The core can be obtained by successively deleting leaves and their incident edges from the graph until no leaves remain.)  The core is empty if there are no cycles. If the core is non-empty, removing those of its edges which do not lie in a cycle yields one or more components, which represent the cycle structure of the original component.

It is possible to define an analogous notion of a core for each component $\mathcal{A}_i, i \in \N$ of the scaling limit of the critical random undirected graph, created by the point-identifications we make in the $\R$-trees encoded by the excursions of $W^{\lambda}$. Indeed, for each $i \ge 1$, $\mathrm{core}(\mathcal{A}_i)$ is a connected undirected multigraph with edge lengths which is empty if there are no point-identifications, is a loop if there is a single point-identification, and is otherwise 3-regular almost surely.  For each $i \ge 1$, if $\mathrm{core}(\mathcal{A}_i)$ is non-empty and we remove from it any points not contained in cycles, we obtain a collection of one or more multigraphs with edge-lengths which are again either loops or 3-regular.  Let us refer to these as the \emph{cycle-components}.  The MDMs $(\mathcal{C}_i, i \in \N)$ are similarly obtained by making (a different collection of) point-identifications in the $\R$-trees encoded by the excursions of $W^{\lambda}$. In this context, a single $\R$-tree may give rise to one or more strongly connected components, or indeed none. The fact that we obtain an $\ell^1$ convergence in Theorem~\ref{thm:main}, comes from the property that for very small $\sigma,$ an $\R$-tree with the same distribution as $\T_{\sigma}$ is very unlikely to produce any strongly connected components at all. 

It would be interesting to know if the distributions of the undirected version of $(\mathcal{C}_i, i \in \N)$ and of the decreasing ordering of all the cycle-components coming from $(\mathrm{core}(\mathcal{A}_i), i \in \N)$ are mutually absolutely continuous. We leave this as an open problem. 

\bigskip

The rest of this paper is structured as follows. In Section~\ref{sec:graphtheory}, we introduce some standard terminology and then describe the depth-first exploration which we use in order to understand the directed graph $\vec{G}(n,p)$.  A key role is played by a particular class of edges known as back edges, and we discuss back edges in both the discrete and continuum settings in Section~\ref{sec:backedges}.  In Section~\ref{sec:scalinglimit}, we prove some useful properties of the scaling limit $\mathcal{C}$.  Section~\ref{sec:mainproof} contains the proof of Theorem~\ref{thm:main}.  In Section~\ref{sec:furtherprops}, we prove the further properties of the scaling limit which were mentioned immediately after the main theorem.

\section{Some graph theory} \label{sec:graphtheory}
\subsection{Basic terminology}
We recall here some elementary graph theoretic terminology which we will use throughout the paper.

\medskip
\noindent\textbf{Directed graphs and strongly connected components.} Let $\vec{G}$ be a directed graph. For a directed edge $(x,y)$ of $\vec{G}$, we say that $x$ is the \emph{tail} of the edge and $y$ is its \emph{head}. For two vertices $x$ and $y,$ we also say that $x$ is a \emph{parent} of $y$ (and $y$ is a \emph{child} of $x$) if there is an edge from $x$ to $y$, and that $x$ is an \emph{ancestor} of $y$ (and $y$ is a \emph{descendant} of $x$) if there is a directed path from $x$ to $y$.

A directed graph $\vec{G}$ is \emph{strongly connected} if for every pair $\{u,v\}$ of distinct vertices of $\vec{G}$ there exists a directed path from $u$ to $v$ and a directed path from $v$ to $u$.  For a general directed graph $\vec{G}$, its \emph{strongly connected components} are the maximal strongly connected subgraphs.  The strongly connected components partition the vertex set but note that, unlike for undirected graphs, edges of $\vec{G}$ may lead from one strongly connected component to another.

\medskip

\noindent\textbf{Trees and plane trees.} A discrete tree is a connected undirected graph $T$ with no cycles. For two vertices $x$ and $y$ in $T$, we write $\llbracket x,y\rrbracket$ for the unique path between $x$ and $y.$ Our trees will often be rooted at a specified vertex $\rho$. This allows us to think of $T$ as a \emph{directed} graph, by orienting all of its edges away from $\rho.$  We write $|T|$ for the size of the vertex set of $T$ and $\|T\|$ for the \emph{height} of $T$, that is the largest distance between $\rho$ and another vertex.

A \emph{planar ordering,} also known as \emph{topological sort}, of a rooted tree $T$ is any total order $>$ on its vertex set such that every directed edge $(u,v)$ of $T$ is increasing, in the sense that $v > u$ (decreasing edges are defined similarly). A rooted plane tree is then a rooted tree endowed with a planar ordering.

\medskip

\noindent\textbf{Directed multigraphs.} 
Recall the definition of a directed multigraph from the introduction. Directed multigraphs have the same notion of ancestor and descendant as directed graphs, and have strongly connected components in the same way. Note that the loop is strongly connected.  The \emph{excess} of a strongly connected directed multigraph $(V,E,r)$ is defined to be $|V| - |E|$.  If the excess is strictly positive then we say that the multigraph is \emph{complex}.

\subsection{The exploration process}\label{sec:exploration}

The strongly connected components of any directed graph can be found in time which is linear in the sum of the sizes of the vertex and edge sets.  Several linear-time algorithms, including Tarjan's algorithm~\cite{Tarjan72} and the so-called path-based algorithms (see \cite{Gabow2000} for an example), rely on a \emph{depth-first search}, that is a procedure which consists in exploring the graph in such a way that, after we visit a vertex, we visit all of its as-yet unseen descendants before backtracking. Broadly speaking, as we traverse the graph, some information is kept in the form of a stack, which allows us to determine the strongly connected components.

For our study of $\vec{G}(n,p),$ we use a variant of these ideas to give a simple algorithm which does not directly yield the strongly connected components, but instead gives a specific \emph{plane} spanning forest which will be a key part of the structure of the strongly connected components. In order to find our spanning forest, we use the now-standard ordered depth-first search exposed, for example, in \cite{A-BBG12}, but with the modification that we only allow ourselves to follow edges in the direction of their orientation. The standard ordering on the vertex set $[n]$ is used to induce a planar ordering on the out-neighbours of a vertex.  Let us give a precise definition of the construction and, along the way, remind the reader of the depth-first exploration for undirected graphs. Let $\vec{G}$ (resp.\ $G$) be any directed graph (resp.\ undirected graph) on $[n].$ Inductively on $i\in \{0,\ldots,n\}$, we define an ordered list $\mathcal{O}_i$ of open vertices (the \emph{stack}) which have been seen but not yet explored, and a set $\mathcal{E}_i$ of explored vertices:

\begin{itemize}
\item[$\bullet$] $i=0$: let $\mathcal{O}_0=(1)$ and $\mathcal{E}_0=\emptyset.$
\item[$\bullet$] Induction step: given $\mathcal{O}_i$ and $\mathcal{E}_i$, let $v_i$ be the first vertex of $\mathcal{O}_i$ and let $\mathcal{E}_{i+1}=\mathcal{E}_i \cup \{v_i\}.$ Let $\mathcal{N}_i$ be the set of out-neighbours (resp.\ neighbours) of $v_i$ which are not in $\mathcal{O}_i\cup \mathcal{E}_i.$ Construct $\mathcal{O}_{i+1}$ by removing $v_i$ from $\mathcal{O}_i$, and adding in the elements of $\mathcal{N}_i$ in increasing order, so that the smallest element of $\mathcal{N}_i$ is now at the start of $\mathcal{O}_{i+1}.$ If, however, this leads to $\mathcal{O}_{i+1}=\emptyset,$ then add to it the smallest element of $\{1,\ldots,n\} \setminus \mathcal{E}_{i+1}.$
\end{itemize}

This procedure builds a directed spanning forest $\mathcal{F}_{\vec{G}}$ of $\vec{G}$, by saying that two vertices $x$ and $y$ are linked by an edge from $x$ to $y$ if there exists $i$ for which $x=v_i$ and $y\in \mathcal{N}_i.$ This is illustrated in Figure~\ref{fig2}, for the graph given by Figure~\ref{fig1}. We call $\mathcal{F}_{\vec{G}}$ the \emph{forward depth-first forest} of $\vec{G}$.

We also obtain a total order of $[n]$, given by $(v_0,\ldots,v_{n-1}),$ which is a planar ordering of $\mathcal{F}_{\vec{G}},$ in the sense that it is a topological sort of each of its trees and it also functions as a total order on the set formed by the trees.  We write $v_i \prec v_j$ iff $i < j$.  The edges of $\vec{G}$ may now be partitioned into two categories: the \emph{forward edges}, which are increasing for this order, and the \emph{back edges}, which are decreasing. The forward edges can themselves also be separated into two sets: those which are edges of $\mathcal{F}_{\vec{G}},$ and those which are not, which we call \emph{surplus} edges. (In the case of the undirected graph $G$, we still get a forest $\mathcal{F}_G$, but all edges of $G$ are either part of the forest or are surplus edges.)  

The combination of forward edges and back edges is what creates the strongly connected components of $\vec{G}.$ Notice in particular that, since there are no forward edges going between different trees of $\mathcal{F}_{\vec{G}},$ each strongly connected component lies within a single such tree. Moreover, since strongly connected components are made of cycles, any strongly connected component with at least two vertices must contain at least one forward and one back edge.  We call an \emph{ancestral} back edge one which goes from a vertex to one of its ancestors.  See Figure~\ref{fig2} for an illustration. The ancestral back edges will play a particularly important role in the sequel.

\begin{lem} \label{lem:surplusorback}
Each strongly connected component contains either a surplus edge or an ancestral back edge.
\end{lem}

\begin{proof}
Suppose there are no surplus edges.  Any strongly connected component has a least element $x_0$ for the depth-first ordering.  Then the vertex $x_0$ can only have in-edges from vertices which are later in the ordering, and must have at least one such.  Take the smallest of these, namely the edge $(x_1,x_0)$ with the smallest $x_1$ for the ordering. There must be a path from $x_0$ to $x_1$.  We claim that this path can use only forward edges of the tree.  Suppose not.  Then $x_1$ does not belong to the subtree rooted at $x_0$.  But then any path from $x_0$ to $x_1$ must use at least one surplus edge, which contradicts the assumption that there are none.  It follows that $x_0$ is an ancestor of $x_1$, and so the edge $(x_1,x_0)$ is ancestral.
\end{proof}

We deduce from this a useful bound: the number of strongly connected components of $\vec{G}$ is smaller than the sum of its numbers of surplus edges and ancestral back edges.

\begin{figure}[h]
\centering
\includegraphics[scale=0.7]{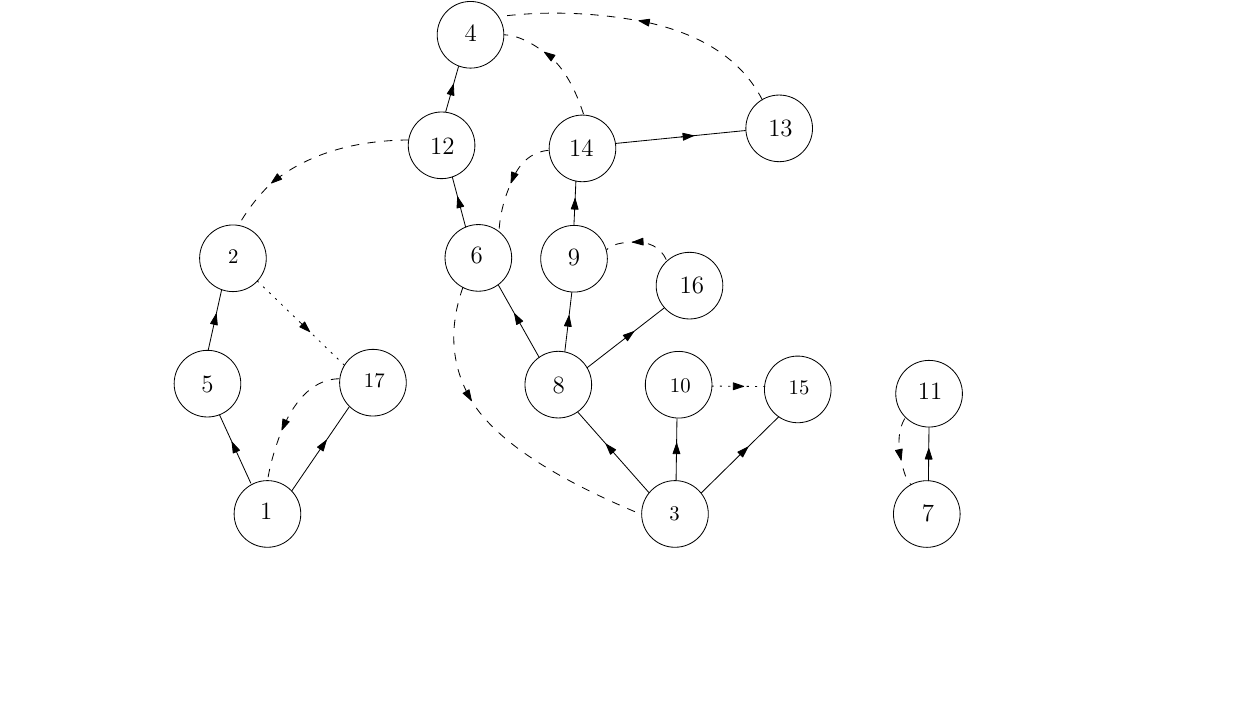}
\caption{The planar embedding of exploration forest of the graph in Figure~\ref{fig1}. Surplus edges and back edges are then dotted, and respectively straight and curved. Notice that there are back edges of both kinds: ancestral, such as $(6,3)$, and non-ancestral, such as $(14,6).$ The strongly connected components have vertex sets $\{3,6,8,9,14,16\}, \{1, 2, 5, 17\}, \{7,11\}, \{4\}, \{10\}, \{12\}, \{13\}$ and $\{15\}$.}
\label{fig2}
\end{figure}
Note that the surplus edges of $G$ are taken from the set of edges \emph{permitted} by $\mathcal{F}_{\vec{G}},$ which are the pairs $(u,v)$ such that there exists $i$ such that $u$ and $v$ are both in $\mathcal{O}_i$.  In this case, $v$ is a sibling of an ancestor of $u$ which occurs later in the planar ordering.\footnote{Note also that $\mathcal{F}_{\vec{G}}$ determines the exploration process fully, so defining the permitted edges using the $\mathcal{O}_i$ is unambiguous.} In fact, given $\mathcal{F}_{\vec{G}},$ we can add or remove any permitted edge to $\vec{G},$ and this will not change $\mathcal{F}_{\vec{G}}.$ The same holds true for back edges. Thus, conditionally on $\mathcal{F}_{\vec{G}(n,p)},$ the permitted surplus edges and back edges of $\vec{G}(n,p)$ appear independently with probability $p$. This leads to the following proposition, which allows us to relate $\vec{G}(n,p)$ to $G(n,p)$ by their explorations.
\begin{prop}\label{prop:coupling} For any directed graph $\vec{G}$ on $[n]$ we call $\vec{G}_{\mathrm{fwd}}$ the undirected graph obtained by removing the back edges of $\vec{G}$ and keeping the forward edges, but ignoring their orientation. We then have the following:
\begin{itemize}
\item[$(i)$] $\mathcal{F}_{\vec{G}(n,p)}\overset{(d)}=\mathcal{F}_{G(n,p)}$
\item[$(ii)$] $(\vec{G}(n,p))_{\mathrm{fwd}}\overset{(d)}=G(n,p)$
\item[$(iii)$] One can couple $G(n,p)$ and $\vec{G}(n,p)$ in the following way: first sample $G(n,p)$, which creates in particular a depth-first ordering on $\{1,\ldots,n\}.$ Then let $(\vec{G}(n,p))_{\mathrm{fwd}}=G(n,p),$ and add to it each of the possible back edges $(v_i,v_j)$ for $j<i$ independently with probability $p$.
\end{itemize}
\end{prop}

\begin{proof}
The proof of \emph{(i)} is straightforward by induction: notice that, in the explorations of both $\vec{G}(n,p)$ and $G(n,p),$ for all $i$, given $\mathcal{O}_i$ and $\mathcal{E}_i$, the neighbourhood $\mathcal{N}_i$ contains each element of $\{1,\ldots,n\}\setminus (\mathcal{O}_i\cup \mathcal{E}_i)$ independently with probability $p$. Thus, each step of the forward exploration of $\vec{G}(n,p)$ has the same distribution as the corresponding step of the depth-first exploration of $G(n,p),$ and in particular the forests they build have the same distribution.

Part \emph{(ii)} is obtained by observing that, both for $\vec{G}(n,p)$ and $G(n,p),$ given the exploration forest, each permitted surplus edge is present independently with probability $p$. Similarly, \emph{(iii)} follows from the fact that, given $\vec{G}_{\mathrm{fwd}}(n,p)$, each back edge is present independently with probability $p$.
\end{proof}

This proposition motivates the study of a process which adds back edges to trees. The next section will formalise this, especially for the continuum trees which arise in the scaling limit of $G(n,1/n + \lambda n^{-4/3}).$

\section{Back edges on discrete and continuum trees} \label{sec:backedges}
We show that, when considering a plane tree with additional back edges, we can safely ignore a portion of the back edges and keep the same strongly connected components. We then adapt this idea to give a procedure for building a random finite set of backward identifications on a continuum tree, which is how $\mathcal{C}$ will be built.
\subsection{The discrete case}\label{sec:discretebackedges}
Let $T = (V(T), E(T))$ be a finite rooted plane tree, with root $\rho$ and $|V(T)| = n$.  We think of $T$ as a directed graph, by orienting all the edges away from $\rho.$  Let us write $(v_0, v_1, \ldots, v_{n-1})$ for the vertices of $T$ listed in the depth-first order explained in Section~\ref{sec:exploration}.  Recall that for two vertices $x$ and $y$ of $T$, $\llbracket x,y\rrbracket$ is the path between $x$ and $y$.

Consider a set $B$ of additional edges between elements of $V(T)$ which go backwards for the depth-first order, so that any element of $B$ is of the form $(v_i, v_j)$ with $v_j \prec v_i$ (i.e.\ $j < i$). Such an edge is called \emph{ancestral} if $v_j$ is an ancestor of $v_i$ in $T$.  It is useful to have an ordering on the elements of $B$, and we do this by declaring that $(v_i, v_j) \prec (v_k, v_l)$ if $i < k$ or if $i=k$ and $j < l$.  (This is just the lexicographic ordering on the pairs of indices $(i,j)$ and $(k,l)$.) Note that the elements of $B$ are thus listed in the order in which we would encounter them when performing a depth-first search of the directed graph $(V(T), E(T) \cup B)$.  

We now extract a subset $B^*=\big\{(x_i,y_i),i\leq N\big\}$ of $B$ inductively as follows (as illustrated in Figure~\ref{fig3}).
\begin{itemize}
\item Let $(x_1,y_1)$ be the first ancestral back edge in $B.$
\item Assume that we are now given $(x_j,y_j)$ for $j \le i.$ For $x\in V(T)$ such that $x \succ x_i$, define $T(i,x)=\bigcup_{j=1}^{i}\llbracket \rho,x_j \rrbracket\cup\llbracket \rho,x\rrbracket$.  Let $(x_{i+1},y_{i+1})$ be the smallest element $(x,y)$ of $B\setminus\big\{(x_j,y_j),j\in \{1,\ldots,i\}\big\}$ such that $y\in T(i,x).$ If there is no such element, then we end the procedure, and set $N=i.$
\end{itemize}

Observe that all of the ancestral edges in $B$ are elements of $B^*$, since if $(x,y)$ is ancestral then $y\in \llbracket \rho,x\rrbracket.$ Moreover, notice that, in the inductive part of the definition of $B^*,$ we could equivalently have defined $(x_{i+1},y_{i+1})$ to be smallest element $(x,y)$ of $B\setminus\big\{(x_j,y_j),j\in \{1,\ldots,i\}\big\}$ which is either ancestral or such that there is a directed path from $y$ towards an ancestral back edge. This implies that $B^*$ is exactly the set of back edges which are either themselves ancestral, or which lead to ancestral back edges.

\begin{prop}\label{prop:starcomponents} Let $X$ be the directed graph obtained by taking $T$ (with edges directed away from $\rho$) and adding all the edges of $B$. Let $X^*$ be the subgraph of $X$ where we remove any element of $B$ that is not in $B^*.$ Then $X$ and $X^*$ have the same strongly connected components.
\end{prop}
\begin{proof} Let $(x,y)$ be an element of $B$ which is in a strongly connected component of $X.$ We want to prove that it is in $B^*.$  Since there are no surplus edges, by Lemma~\ref{lem:surplusorback}, the strongly connected component must contain an ancestral back edge. Thus, it is possible to reach that ancestral edge starting from $y$ and following edges of $X$. Hence, $(x,y)$ is in $B^*$.  It follows that $X$ and $X^*$ have the same strongly connected components.
\end{proof}

\begin{figure}[h]
\centering
\includegraphics[scale=0.7]{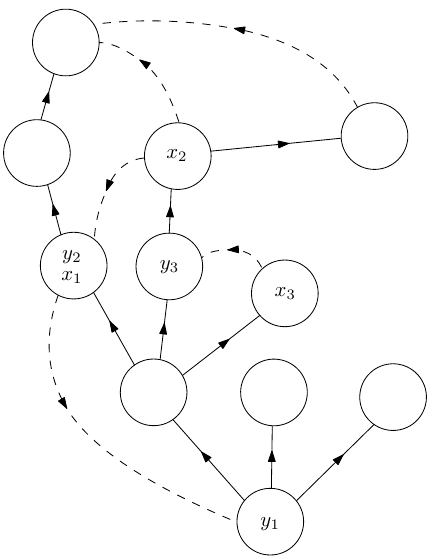}
\caption{The construction of Section~\ref{sec:discretebackedges} applied to the second tree of Figure~\ref{fig2}. Notice that the two topmost back edges are not contained in a strongly connected component, and they are not of the form $(x_i,y_i)$.}
\label{fig3}
\end{figure}

This seemingly innocuous lemma is, in fact, a key tool for us. Indeed, if $T$ is taken to be a large tree of $\mathcal{F}_{\vec{G}(n,p)}$ (meaning it has size of order $n^{2/3}$), and $B$ is the set of back edges of $\vec{G}(n,p)$ which join elements of $T$, then $B$ has size of order $n^{1/3}$. However, as we will see later, the number of back edges in $X^*$ remains of order $1,$ and in fact the $(x_i,y_i)$ will, in the scaling limit, converge to continuous analogues. This means that the reduction from $X$ to $X^*$, while not changing the strongly connected components, allows us to ignore the majority of the back edges at no cost.  (Let us emphasise, however, that in general there are back-edges in $X^*$ which are \emph{not} contained in any strongly connected component.)

\subsection{The continuum case} \label{sec:continuum}
\subsubsection{$\R$-trees and notation} \label{subsec:rtrees}
We recall here some basic terminology about $\R$-trees; more information concerning their use in probability may be found in the survey paper \cite{LG2005}. An \emph{$\R$-tree} is any metric space $(\T,d)$ such that
\begin{itemize}
\item For all $x,y \in \T$, there exists a unique distance-preserving map $\phi_{x,y}$ from $[0,d(x,y)]$ into $\mathcal{T}$ such  $\phi_{x,y}(0)=x$ and $\phi_{x,y}(d(x,y))=y.$ We write $\llbracket x,y\rrbracket$ for the image of $\phi_{x,y}.$

\item For all continuous and one-to-one functions $c$: $[0,1] \to \T $, we have $\\ c([0,1]) = \llbracket c(0),c(1)\rrbracket.$
\end{itemize}
Our $\R$-trees will typically be \emph{rooted}, which means we distinguish a point of $\T$ called the root, usually denoted $\rho.$ For $x,y\in\T$, we say that $x$ is an \emph{ancestor} of $y$, or that $y$ is a \emph{descendant} of $x,$ if $x\in \llbracket \rho,y\rrbracket,$ and we call the point $x\wedge y$ such that $\llbracket\rho,x\rrbracket\cap\llbracket\rho,y\rrbracket=\llbracket\rho,x\wedge y\rrbracket$ the \emph{most recent common ancestor of $x$ and $y$}. The \emph{degree} of a point $x \in \mathcal{T}$ is the number of connected components of $\mathcal{T} \setminus \{x\}$. If $x$ has degree 1 we call it a \emph{leaf}.  The \emph{height} $\|\mathcal{T}\|$ of $\T$ is the largest distance from $\rho$ to another point.

The $\R$-trees we encounter will all be encoded by functions. For $\sigma\in(0,\infty),$ a function $f:[0,\sigma]\to \R_+$ is called an \emph{excursion function} if it is continuous and $f(x)=0$ if and only if $x=0$ or $\sigma.$  Let $\hat{f}:[0,\sigma]^2 \to \R_+$ be defined by $\hat{f}(x,y)=\underset{s\in[x\wedge y,x\vee y]}\min f(s).$ Then $f$ encodes a pseudo-distance $d_f$ on $[0,\sigma]$, defined by $d_f(x,y)=f(x)+f(y)-2\hat{f}(x,y),$ and an $\R$-tree $\T_f,$ defined by
\[\T_f = [0,\sigma]/\{d_f=0\}.\]
The {natural projection} from  $[0,\sigma]$ to $\T_f$ will be called $p_f$, and we let the root of $\T_f$ be $p_f(0)=p_f(\sigma).$ $\T_f$ also inherits a natural total order from $[0,\sigma]$ which we call the \emph{planar order.} 

In the sequel, we will assume that the functions $f$ we consider have unique local minima.  Under this assumption, the resulting $\R$-tree $\mathcal{T}_f$ is binary (meaning its points all have degree at most 3). We also assume that the push-forward of the Lebesgue measure on $[0,\sigma
]$ onto $\mathcal{T}_f$ via $p_f$ is concentrated on the leaves so that, in particular, if $U$ is uniform on $[0,\sigma]$ then $p_f(U)$ is a leaf with probability 1. These properties hold almost surely for the \emph{random} excursion functions we will use in Subsection~\ref{subsec:boundssingletree}.



\subsubsection{Constructing the identifications}\label{sec:continuousbackedges}

We now describe a random process which will give us a finite number of point identifications which go backward for this planar ordering: pairs of points of the form $(x,y)$ with $x>y$ for the planar ordering and an ``arrow with zero length" pointing from $x$ to $y$. Specifically, we will define times $(s_i,i\in \{1,\ldots,N\})$ in $[0,\sigma]$, their projections $x_i=p_f(s_i)$ and points $y_i$ such that $y_i$ is in the subtree $\T_i=\bigcup_{j=1}^{i}\llbracket \rho,x_j \rrbracket.$

We start with the base case $i=1.$ Let $s_1$ be the first point of a Poisson point process on $[0,\sigma]$ with intensity $f(x)\mathrm d x$, and $x_1=p_f(s_1).$ Then let $y_1$ be a uniform random point on the segment $\llbracket \rho,x_1\rrbracket.$ If the Poisson point process has no points on $[0,\sigma],$ we let $N=0.$

Now let us assume that $(s_j,x_j,y_j)$ have been built for $j\in\{1,\ldots,i\}.$ For $t\in [s_i,\sigma],$ consider the subtree $\T_i(t)=\T_i \cup \llbracket \rho,p_f(t)\rrbracket$ and its length $\ell_i(t).$ Then, straightforwardly,

\begin{equation}\label{eq:rate}
\ell_i(t)=\sum_{j=1}^i \left(f(s_j)-\hat{f} (s_{j-1},s_j) \right) + f(t)-\hat{f}(t,s_i).
\end{equation}
Now let $s_{i+1}$ be the first arrival time of an independent Poisson point process on $[s_i,\sigma]$ of intensity $\ell_i(t)\mathrm d t,$ let $x_{i+1}=p_f(s_{i+1}),$ and let $y_{i+1}$ be uniform on the finite length space $\T_i(s_i),$ independently of everything else. If the Poisson point process has no points, we let $N=i.$


\begin{figure}[h]
\begin{center}
\includegraphics[scale=0.8]{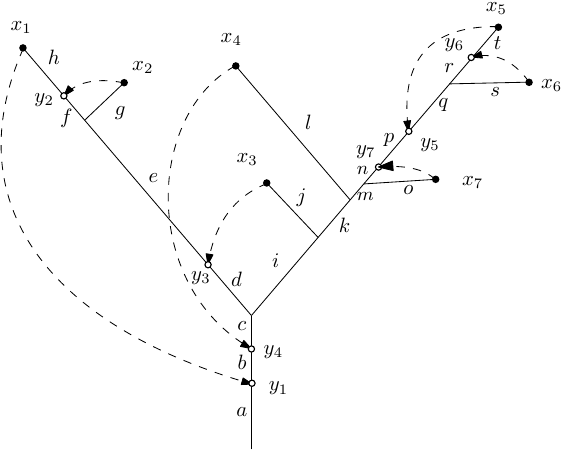}
\end{center}
\caption{Some identifications on a tree. Lengths of the segments are represented by $a,\ldots,t$ for the next figures.}
\label{fig:twogen}
\end{figure}

Observe that we necessarily have $s_1 \leq s_2 \leq \ldots \leq \sigma$ (so that, in particular, the first back edge in the planar ordering is always ancestral).  It is, however, in principle possible for the sequence $(s_i)_{k \ge 1}$ to accumulate, with the consequence that there are infinitely many back edges. In that case, we set $N=\infty.$ However, this in fact occurs with probability $0$. Indeed, by (\ref{eq:rate}), we have that $\ell_i(t) \le (i+1) \| f \|,$ with $\| \cdot \|$ denoting the usual supremum norm. Thus, for all $i\in\N$ and $t\geq 0,$ 
\[\pr[N\geq i \text{ and } s_i-s_{i-1}\leq t \mid N\geq i-1] \leq \pr [E_i \leq t]\]
where $(E_i,i\in\N)$ are independent exponential variables with respective parameters $(i \|f\|,i\in\N),$ and $s_{0}=0$ by convention. Hence we also have
\[\pr[N=\infty]=\pr \left[N=\infty \text{ and } \sup_{i\in\N}s_i\leq \sigma \right]\leq \pr \left[\sum_{i=1}^{\infty} E_i \leq \sigma \right]. \]
However, $\sum_{i=1}^{\infty} E_i=\infty$ a.s.\ by the divergence of the harmonic series and standard properties of sums of exponential variables (see e.g. Theorem 2.3.2 of \cite{NorrisMarkovChains}), and so the above probability is $0.$

%

We end this section with two elementary technical points.

\begin{lem}\label{lem:explicit} The joint distribution of $N$ and the $(s_i,i\leq N)$ can be written explicitly as:
\begin{align*}
& \pr[N=n,s_i\in\mathrm{d}t_i \,\forall i\leq n]= \\
&\qquad \qquad \prod_{k=1}^n \left(\sum_{i=1}^k \left( f(t_i) - \hat{f} (t_{i-1},t_i) \right) \right) \\ 
    & \qquad \qquad \times \exp{\left(-\int_0^{\sigma} \Bigg(f(t)-\hat{f}(t_{I(t)},t) + \sum_{i=1}^{I(t)}\left(f(t_i)-\hat{f}(t_{i-1},t_i) \right)\Bigg)\mathrm{d}t\right)} \mathrm{d}t_1\ldots\mathrm{d}t_n,
\end{align*}
where $t_0=0$ and, for $t\in[0,\sigma]$, $I(t)=\max\{i: t_i<t\}.$ 
\end{lem}

\begin{rem}\label{rem:otherconstruction} It will at times be convenient to consider the pairs $(x_i,y_i)$ which are \emph{ancestral}, i.e. such that $y_i$ is an ancestor of $x_i.$ Thus we let $\big((x_i^a,y_i^a),i\leq N^a\big)$ be those pairs, and $(s_i^a,i\leq N_a)$ the corresponding times in $[0,\sigma].$ Note that these are the points of a Poisson point process with intensity $f(x)\mathrm dx$ on $[0,\sigma].$ In particular, $N^a$ is a Poisson variable with parameter $\int_0^{\sigma} f(t)\mathrm d t,$ and conditionally on $N^a=1,$ the time $s_1^a$ has density proportional to $f$ on $[0,\sigma].$

On the other hand, if we let $\big((x_i^b,y_i^b),i\leq N^b\big)$ be the pairs which are \emph{not} ancestral, then the corresponding times $(s_i^b,i\leq N^b)$ can, conditionally on $(s_i^a,i\leq N^a),$ also be seen as the points of a Poisson point process. We will not need this full description, but the following expression will be useful:
\[\pr[N^b=0 \mid N^a=1,s_1^a]=\exp\left(-\int_{s_1^a}^{\sigma} \left( f(t)-\hat{f}(s_1^a,t) \right) \mathrm d t\right). \]
\end{rem}

\subsubsection{The resulting strongly connected components}

Let $\T_f^{\mathrm{mk}}=\mathcal{T}_N$ be the subtree spanned by the root and the marked leaves, and quotient it by the equivalence relation $\sim$ which identifies $x_i$ and $y_i$ for $i\in \{1,\ldots,N\}$, to obtain a rooted metric space
\[
\mathcal{M}_f=\T_f^{\mathrm{mk}}/\sim.
\] 
Since $\T_f^{\mathrm{mk}}$ has only finitely many leaves, we may also view $\mathcal{M}_f$ as a finite rooted directed multigraph $M_f$ whose edges are endowed with lengths: the vertices of $M_f$ are the images of the $(y_i)$ and of the branchpoints of $\T_f$, and the directions are inherited from $\T_f^{\mathrm{mk}}$ (which we always think of as having edges directed away from the root). We observe that, with the exception of the root (which can be a leaf), the vertices of $M_f$ all have degree at least 3.  Now remove all edges which do not lie in a strongly connected component of $M_f$ and delete any isolated vertices thus created.  This yields a collection of strongly connected components of minimum degree 2.  If there remain vertices of degree precisely two, we repeatedly apply the following merging operation.  Pick an arbitrary vertex of degree 2 and merge its two incident edges as long as they are different edges, summing their lengths. This yields a collection $\mathcal{C}_f$ of strongly connected MDMs, as illustrated in Figure~\ref{fig:strongconnect}.

\begin{figure}[h]
\begin{center}
\includegraphics[scale=0.7]{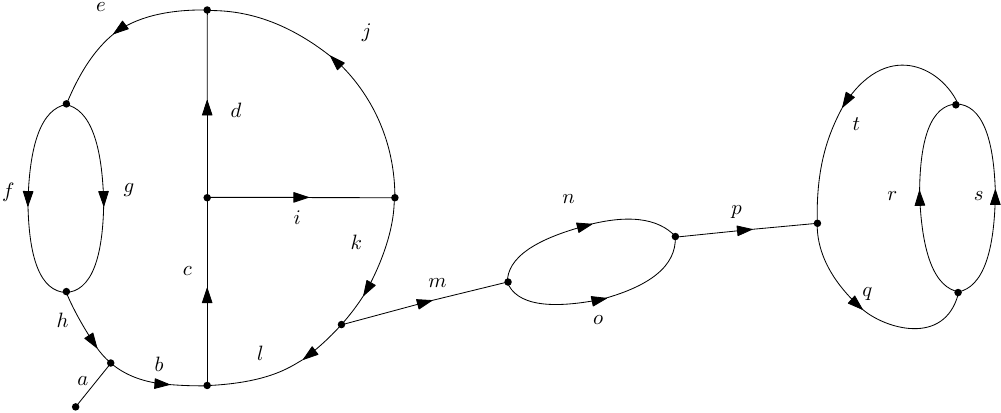}
\vspace{1cm}

\includegraphics[scale=0.7]{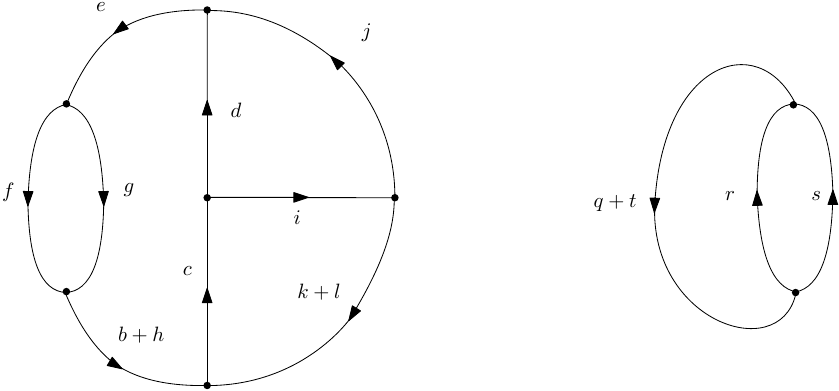}
\caption{A representation of $\mathcal{M}_f$ obtained from the identifications given in Figure~\ref{fig:twogen}, and the resulting strongly connected components.} \label{fig:strongconnect}
\end{center}
\end{figure}

\section{The scaling limit} \label{sec:scalinglimit}
\subsection{Excursions of Brownian motion with parabolic drift}
Let $(W(t),t\geq0)$ be a standard Brownian motion. For $\lambda\in \R$ and $t\geq0,$ let $W^{\lambda}(t)=W(t)+\lambda t -t^2/2$ and let $\underline{W}^{\lambda}(t)=\underset{0\leq s\leq t} \inf W^{\lambda}(s).$ Let $B^{\lambda}(t)=W^{\lambda}(t)-\underline{W}^{\lambda}(t)$, and let $\Gamma^{\lambda}$ be the set of excursions of $B^{\lambda}.$ For an excursion $\gamma \in \Gamma^{\lambda}$, let $|\gamma|$ denote its length.

\begin{prop} \label{prop:excursionlengths}
\begin{itemize}
\item[(i)]For $\alpha\in\{2,3\}$, we have $\E\left[\sum_{\gamma\in\Gamma^{\lambda}} |\gamma|^{\alpha}\right]< \infty$.
\item[(ii)] $\sum_{\gamma\in\Gamma^{\lambda}} |\gamma|^{3/2}=\infty$ a.s.
\end{itemize}
\end{prop}

The $\alpha = 2$ case of $(i)$ is Lemma 25 of Aldous~\cite{aldous1997}, which we extend here to $\alpha=3$.  (Our method of proof also works for all $\alpha>3/2$  but we omit the details for the sake of brevity.)  We first need a standard result on moments of hitting times of Brownian motion with constant drift.

\begin{lem} For $\mu>0$ and $b>0$, let $T(b,\mu)=\inf\{t\geq0:\, W(t)-\mu t = -b\}.$ Then we have
\[\E[T(b,\mu)]=\frac{b}{\mu} \quad\text{and}\quad \E \left[\big(T(b,\mu)\big)^2\right]=\frac{b(1+ b \mu)}{\mu^3}.\]
\end{lem}

\begin{proof} The Laplace transform of $T(b,\mu)$  is given by
\[
\E\left[e^{-\theta T(b,\mu)}\right]=\exp\left( b\mu - b\sqrt{\mu^2+2\theta}\right), \:\theta>0,
\]
(see, for example, Exercise 5.10 in Chapter 3 of \cite{KaratzasShreve}) and the first two moments of $T(b,\mu)$ follow from differentiating twice.
\end{proof}

\begin{proof}[Proof of Proposition~\ref{prop:excursionlengths} $(i)$.]  We adapt the proof of Lemma 25 of Aldous~\cite{aldous1997}.  The time-interval $[0,\infty)$ may be split into the union of a countable collection of open intervals during each of which the reflected Brownian motion with drift $B^{\lambda}$ is away from zero (the \emph{excursion intervals}), and the complementary set of times (when the process is at 0), which has zero Lebesgue measure.  (For a careful proof of these excursion-theoretic facts, see Section 3.2 of Goldschmidt and Conchon-Kerjan~\cite{C-KG}.) Let $\gamma$ be an excursion of $B^{\lambda},$ and let $l$ and $r$ be its endpoints. We have
\[|\gamma|^{3}=3 \int_l^r(r-t)^{2} \mathrm d t.\]
For $t\geq 0$, write $H_t=\min \{s>0: B^{\lambda}(t+s)=0\}$.  Then we have
\[\sum_{\gamma\in\Gamma^{\lambda}} |\gamma|^{3}=3\int_0^{\infty}H_t^{2}\mathrm d t\]
and so we only need to prove that  $\int_0^{\infty} \E[H_t^2]\mathrm d t<\infty.$ To do this we split the integral into $\int_0^{\tau} \E[H_t^2]\mathrm d t$ and $\int_{\tau}^{\infty} \E[H_t^2]\mathrm d t$ where $\tau= \max (0,2\lambda+1).$

For $t >\max(0,\lambda)$ and $s\geq 0$, we have $\lambda s - (t+s)^2/2 + t^2/2 \leq (\lambda-t)s.$ Thus, conditionally on $B^{\lambda}(t)$, the process $(B^{\lambda}(t+s),s\geq0)$ may, until it hits $0$, be coupled with a Brownian motion with initial value $B^{\lambda}(t)$ and drift $\lambda-t,$ in such a way that the latter is larger than or equal to the former. Hence $H_t$ is stochastically dominated by $T(B^{\lambda}(t), t - \lambda),$ which leads to
\[\E[H_t\mid B^{\lambda}(t)]\leq\frac{B^{\lambda}(t)}{t-\lambda} \quad\text{and}\quad \E[(H_t)^2\mid B^{\lambda}(t)]\leq\frac{B^{\lambda}(t)(1+(t-\lambda) B^{\lambda}(t))}{(t-\lambda)^3}.\]
In particular, we have 
\[\int_{\tau}^{\infty} \E[H_t^{2}]\mathrm d t\leq \int_{\tau}^{\infty} \frac{\E[B^{\lambda}(t)]+(t-\lambda)\E\left[\left(B^{\lambda}(t)\right)^2\right]}{(t-\lambda)^3}\mathrm d t.\]
However, it is also established in the proof of Lemma 25 of \cite{aldous1997} that, for $t>2\lambda,$ the random variable $B^{\lambda}(t)$ is stochastically dominated by an exponential random variable with parameter $t-2\lambda$, implying that $\E[B^{\lambda}(t)]\leq 1/(t-2\lambda)$ and $\E[(B^{\lambda}(t))^2]\leq 2/(t-2\lambda)^2.$ In consequence, $\int_{\tau}^{\infty} \E[H_t^{2}]\mathrm d t<\infty.$

To bound $\int_0^{\tau} \E[H_t^2]\mathrm d t,$ notice that we have $H_t \leq \tau-t + H_{\tau}\leq \tau + H_\tau$ for $t\leq \tau$. Hence,
\begin{align*}
\E[H_t^2]&\leq \tau^2+2\tau\E[H_{\tau}]+\E[H_{\tau}^2]<\infty \\
&\leq \tau^2+2\tau\frac{\E[B^{\lambda}(\tau)]}{\tau-\lambda}+\frac{\E[B^{\lambda}(\tau)]}{(\tau-\lambda)^3}+\frac{\E[(B^{\lambda}(\tau))^2]}{(\tau-\lambda)^2},
\end{align*}
and this uniform upper bound is finite since $B^{\lambda}(\tau)$ is stochastically dominated by an exponential variable which has moments of all orders. It follows that we do indeed have $\int_0^{\tau} \E[H_t^2]\mathrm d t<\infty.$
\end{proof}

Part $(ii)$ also requires a preliminary lemma, which allows us to work directly with $B^{\lambda}$ instead of powers of its excursion lengths.
\begin{lem}\label{lem:sumintegral}
If $\int_0^{\infty} B^{\lambda}(t) \mathrm{d} t =\infty$ a.s.\ then $\sum_{\gamma\in\Gamma^{\lambda}} |\gamma|^{3/2}=\infty$ a.s.
\end{lem}
\begin{proof}
Let $(\sigma_1, \sigma_2, \ldots)$ be the lengths of the excursions of $B^{\lambda}$ listed in decreasing order, and let $\tilde{e}_1, \tilde{e}_2, \ldots$ be the excursions themselves, so that
\[
\int_0^{\infty} B^{\lambda}(t) \mathrm{d}t = \sum_{i \ge 1} \int_0^{\sigma_i} \tilde{e}_i(x) \mathrm{d} x.
\]
Now, we have by~\cite[Section 5]{A-BBG12} that $\tilde{e}_i \overset{d}{=} \tilde{\exc}^{(\sigma_i)}$ for $i \ge 1,$ and the excursions $(\tilde{e}_i, i \ge 1)$ are conditionally independent given their lengths. Moreover, calling $\exc$ a normalised Brownian excursion, we have using Brownian scaling
\[
\E\left[\int_0^{s} \tilde{\exc}^{(s)}(x) \mathrm{d} x\right]=s^{3/2}\frac{\E\left[\int_0^1\exc(x)dx\exp \left(s^{3/2} \int_0^1\exc(x)dx \right)
\right]}{\E\left[\exp \left( s^{3/2}\int_0^1\exc(x)dx \right)\right]}\underset{s \to 0}\sim s^{3/2}\E\left[\int_0^1\exc(x)dx\right].
\]
It follows that, almost surely
\[
\E \left[ \sum_{i \ge 1} \int_0^{\sigma_i} \tilde{e}_i(x) \mathrm{d} x \ \Bigg| \ (\sigma_j)_{j \ge 1} \right] = \sum_{i \ge 1} \E \left[ \int_0^{\sigma_i} \tilde{\exc}^{(\sigma_i)}(x) \mathrm{d} x \Big| \sigma_i \right] < \infty 
\]
if and only if
\[
\sum_{i \ge 1} \sigma_i^{3/2} \E \left[\int_0^{1} \exc(x) \mathrm{d} u \right] < \infty, 
\]
which itself occurs if and only if $\sum_{i \ge 1} \sigma_i^{3/2} < \infty$.  But by assumption we have $\int_0^{\infty} B^{\lambda}(t) \mathrm{d} t = \infty$ a.s., and so $\sum_{\gamma \in \Gamma^{\lambda}} |\gamma|^{3/2} = \sum_{i \ge 1} \sigma_i^{3/2} = \infty$ a.s.
\end{proof}

The proof of Proposition~\ref{prop:excursionlengths} $(ii)$ is due to \'Eric Brunet.

\begin{proof}[Proof of Proposition~\ref{prop:excursionlengths} $(ii)$.] Recall that
\[
W^{\lambda}(t)=W(t)+\lambda t -t^2/2,
\]
$\underline{W}^{\lambda}(t)=\underset{0\leq s\leq t} \inf W^{\lambda}(s)$ and $B^{\lambda}(t)=W^{\lambda}(t)-\underline{W}^{\lambda}(t)$.
By Lemma~\ref{lem:sumintegral}, it is sufficient to show that $\int_0^{\infty} B^{\lambda}(t) \mathrm{d} t=\infty$ a.s.  We will construct a lower bound for $B^{\lambda}$, built on the same probability space, whose integral we can more easily show to be infinite.

Let $(Z(t),t\geq0)$ be  defined by $Z(0)=0$ and
\[
\mathrm d Z(t) = \mathrm d W(t) - \left(3t + \frac{Z(t)}{t}\right) \mathrm{d} t.
\]
Now define $\theta: \R_+ \to \R_+$ by $\theta(s) = (3s)^{1/3}$ and let
\[
Y(s)=\theta(s) Z(\theta(s)).
\]
Then some stochastic calculus gives that
\begin{align*}
\mathrm{d} Y(s) &= \theta'(s)Z(\theta(s)) +\theta(s)\left(\mathrm d W(\theta(s))-\left(3\theta(s)+\frac{Z(\theta(s))}{\theta(s)}\right)\right)\theta'(s)\mathrm d s+\mathrm d{\Big\langle \theta(\cdot),Z(\theta(\cdot))\Big\rangle}_s\\
 &=\theta(s) \mathrm{d} W(\theta(s)) - 3 \mathrm{d}s.
\end{align*}
Note that the first equality used Ito's Lemma and the second used $(\theta'(s))^2\theta(s)=1$ as well as $\mathrm d\big{\langle \theta(\cdot),Z(\theta(\cdot))\big\rangle}_s=0.$ Again since $(\theta'(s))^2\theta(s)=1$, by the Dubins--Schwarz theorem (\cite[Theorem 1.6]{RevuzYor}) we get that 
\[
\widetilde{W}(s) := \int_0^s \theta(r) \mathrm{d} W(\theta(r))
\]
is a standard Brownian motion and so $Y$ is a Brownian motion with drift $-3$:
\[
Y(s)=\widetilde{W}(s)-3s, \quad s \ge 0.
\]
As a consequence, there exists a random time $S_1\geq 0$ such that, for all $s\geq S_1,$ 
\[
-\frac{9}{2}s < Y(s) < 0.
\]
Letting $T_1=\theta(S_1),$ since $\theta^{-1}(t) = t^3/3$, we get 
\[
-\frac{3}{2} t^2<Z(t)<0
\]
for $t \geq T_1.$  In particular, for $t\geq T_1 \vee 2 |\lambda|$, we have the following bounds on the drift term in the SDE defining $Z$:
\[
-\left(3t + \frac{Z(t)}{t} \right) \le -\frac{3}{2} t \le \lambda - t.
\]
Using the fact that $Z$ and $W^{\lambda}$ are constructed from the same Brownian motion $W$, we get that after time $T_2 := T_1\vee 2|\lambda|$, $Z$ has smaller increments than $W^{\lambda}.$

Now choose a time $T_3>T_2$ large enough such that the minima of $Z$ and $W^{\lambda}$ on $[0,T_3]$ are both attained after $T_2.$ Then, for $t\geq T_3$, the minimum of $W^{\lambda}$ on $[0,t]$ is attained at some $u \in [T_2,t]$, and so
\[
B^{\lambda}(t)=W^{\lambda}(t)- \inf_{u \in [T_2,t]} W^{\lambda}(u) = \sup_{u \in [T_2,t]} (W^{\lambda}(t) - W^{\lambda}(u)) \geq \sup_{u \in [T_2,t]} (Z(t)-Z(u)) \geq Z(t)-\underline{Z}(t),
\]
where $\underline{Z}$ is the running infimum of $Z$. Moreover, using the fact that $Y(s)<0$ for $s\geq S_1,$ we get
\[
Z(t)-\underline{Z}(t)=\frac{1}{t}Y(t^3/3) - \underset{u \in[T_2,t]}\inf\frac{1}{u}Y(u^3/3) \geq \frac{1}{t}Y(t^3/3) - \frac{1}{t}\underset{u \in[T_2,t]}\inf Y(u^3/3)=\frac{1}{t}(Y(t^3/3)-\underline{Y}(t^3/3)),
\]
where $\underline{Y}$ is the running infimum of $Y$.  This yields $B^{\lambda}(t)\geq \frac{1}{t}(Y(t^3/3)-\underline{Y}(t^3/3))$ for $t\geq T_3.$ In particular,
\[
\int_0^{\infty} B^{\lambda}(t)\mathrm d t \geq \int_{T_3}^{\infty} \left(Y(t^3/3)-\underline{Y}(t^3/3) \right) \frac{\mathrm d t}{t}=\int_{\frac{1}{3}T_3^3}^{\infty}(Y(s) -\underline{Y}(s))\frac{\mathrm d s}{3s}.
\]

We now show that the final integral is infinite.  Observe that the reflected drifting Brownian motion $Y - \underline{Y}$ is a positive recurrent Markov process, its stationary distribution being exponential with parameter $6$ (see, for example, p.94 of Harrison~\cite{Harrison}). In particular, the sets  of times at which it is above $2$ and below $1$ are both unbounded, and we can define two sequences $(\tau_n)_{n \ge 0}$ and $(\eta_n)_{n \ge 1}$ of stopping times as follows.  Let $\tau_0 = \frac{1}{3}T_3^3$ and then, for $n \ge 1$, 
\begin{align*}
\eta_n & = \inf\{s > \tau_{n-1}: Y(s) - \underline{Y}(s) > 2\}, \\
\tau_n & = \inf\{s > \eta_n: Y(s) - \underline{Y}(s) < 1\}.
\end{align*}
On a downcrossing interval $[\eta_n, \tau_n]$, we trivially have $Y(s) - \underline{Y}(s) \ge 1$, so
\[\int_{\frac{1}{3}T_3^3}^{\infty}(Y(s) -\underline{Y}(s))\frac{\mathrm d s}{s}\geq \sum_{n=1}^{\infty}\frac{(\tau_n - \eta_n)}{\tau_{n}}.
\]
By the strong Markov property, we have that $\{\tau_n - \tau_{n-1}: n \ge 2\}$ are i.i.d., and their expectation is finite by the aforementioned positive recurrence of $Y-\underline{Y}.$ The law of large numbers then implies that, as $n \to \infty$,
\[
\frac{\tau_n}{n} \to \E[\tau_2 - \tau_1] < \infty \quad \text{a.s.}
\]
We also have that $\{\tau_n - \eta_n: n \ge 2\}$ are i.i.d.\ and so $\sum_{n \ge 1} (\tau_n - \eta_n)/n$ diverges a.s.  It follows that
\[
\sum_{n \ge 1} \frac{(\tau_n - \eta_n)}{\tau_n} = \infty \quad \text{a.s.} \qedhere
\]
\end{proof}

\subsection{Bounds for a single tree} \label{subsec:boundssingletree}
Let $\sigma>0.$ We now let $f=2\tilde{\exc}^{(\sigma)}$ be twice a tilted Brownian excursion with length $\sigma$, whose distribution is determined by
\[
\E[g(\tilde{\exc}^{(\sigma)})]=\frac{\E\left[g\big(\sqrt{\sigma}\exc(\cdot/\sigma)\big) \exp \left(\sigma^{3/2}\int_0^1\exc(x)dx \right) \right]}{\E\big[\exp \left( \sigma^{3/2}\int_0^1\exc(x)dx \right)\big]},
\]
for any non-negative measurable function $g$, where $\exc$ is a standard Brownian excursion. The choice is guided by the fact mentioned in the introduction, that these functions (including the factor 2) encode the limiting continuum random trees in which we are interested. We perform the construction detailed in Section~\ref{sec:continuum}, defining the $\R$-tree $\T_{\sigma}$ (we now replace the subscript $f$ by $\sigma$ since henceforth all of our coding functions will be of this type), performing $N_{\sigma}$ identifications, $N_{\sigma}^a$ of them being ancestral and $N_{\sigma}^b$ being non-ancestral, and thus build the MDM $\mathcal{M}_{\sigma}.$ The following proposition will enable us to control the number of strongly connected components of $\mathcal{M}_{\sigma}.$

\begin{prop}\label{prop:Poissonbounds} Let $c=\E[\int_0^1 \mathbf{e}(t)\mathrm dt]=F'(0)$ where $F(z)=\E[e^{z\int_0^1 \mathbf{e}(t)\mathrm d t}]$ is the moment generating function of the Airy distribution, which is an entire function \cite{janson2007}. We have the following asymptotics: as $\sigma \to 0$, 
\begin{itemize}
\item[(i)] $\pr[N_{\sigma}^a=0]= 1 - 2c\sigma^{3/2} + O(\sigma^3)$
\item[(ii)] $\pr[N_{\sigma}^a=1,N_{\sigma}^b=0]= 2c\sigma^{3/2} + O(\sigma^3)$
\item[(iii)] $\pr[N_{\sigma}^a\geq 2\text{ or } N_{\sigma}^b\geq 1]= O(\sigma^3).$
\end{itemize}
Moreover,
\begin{itemize}
\item[(iv)] $\underset{\sigma>0}\sup \; \sigma^{-3}\E[N_{\sigma}^a\mathbf{1}_{\{N_{\sigma}^a\geq 2\}}] <\infty.$
\end{itemize}
\end{prop}

\begin{proof}
Instead of working with $\tilde{\exc}^{(\sigma)},$ we express the probabilities in terms of a standard Brownian excursion $\mathbf{e}$ and its area $\mathcal{A}=\int_0^1 \mathbf{e}(t)\mathrm dt.$ For $(i)$, recall that, conditionally on $\tilde{\exc}^{(\sigma)},$ we have that $N_{\sigma}^a$ has a Poisson distribution with parameter $\int_0^{\sigma} 2\tilde{\exc}^{(\sigma)}(x)\mathrm dx$.  Therefore,
\[
\pr[N_{\sigma}^a=0]=\frac{\E\left[e^{-\sigma^{3/2}\int_0^1 2 \mathbf{e}(t)\mathrm dt} e^{\sigma^{3/2}\int_0^1  \mathbf{e}(t)\mathrm dt}\right]}{\E\left[e^{\sigma^{3/2}\int_0^1  \mathbf{e}(t)\mathrm dt}\right]} 
	=\frac{F(-\sigma^{3/2})}{F(\sigma^{3/2})}=1 - 2c\sigma^{3/2} + O(\sigma^3). \]

We begin the proofs of $(ii)$ and $(iii)$ by computing
\[
\pr[N_{\sigma}^a=1]=\frac{\E\left[2\sigma^{3/2}\mathcal{A}e^{-\sigma^{3/2}\mathcal{A}}\right]}{\E\left[e^{\sigma^{3/2}\mathcal{A}}\right]}=\frac{2\sigma^{3/2}F'(-\sigma^{3/2})}{F(\sigma^{3/2})}=\frac{2\sigma^{3/2}(c+O(\sigma^{3/2}))}{1+O(\sigma^{3/2})}=2c\sigma^{3/2} + O(\sigma^3).
\]
Next, recalling Remark~\ref{rem:otherconstruction}, we write
\begin{align*}
\pr[N_{\sigma}^a=1,N_{\sigma}^b=0]&= \E\left[e^{-2\int_0^{\sigma}\tilde{\exc}^{(\sigma)}(t)\mathrm dt}\int_{0}^{\sigma}2\tilde{\exc}^{(\sigma)}(x) \exp\left(-\int_x^{1}2\left(\tilde{\exc}^{(\sigma)}(y)-\hat{\tilde{\exc}}^{(\sigma)}(x,y)\right)\mathrm dy\right)\mathrm dx \right] \\
                                  &=\frac{1}{{\E[e^{\sigma^{3/2}\mathcal{A}}]}}\E\left[e^{-\sigma^{3/2}\mathcal{A}}\sigma^{3/2}\int_0^1 2\mathbf{e}(x) \exp \left(-\sigma^{3/2}\int_{x}^1 2\left(\mathbf{e}(y)-\hat{\mathbf{e}}(x,y)\right)\mathrm dy\right) \mathrm d x \right].
\end{align*}
Using $|1-e^{-u}|\leq u$ for $u\geq 0,$ we obtain
\begin{align*}
\pr[N_{\sigma}^a=1]-\pr[N_{\sigma}^a=1,N_{\sigma}^b=0]&\leq 4\sigma^3\E\left[e^{-\sigma^{3/2}\mathcal{A}}\int_0^1 \mathbf{e}(x)\mathrm dx\int_x^1(\mathbf{e}(x)-\hat{\mathbf{e}}(x,y))\mathrm d y \right] \\
                              &\leq 4\sigma^{3}\E\left[\int_0^1 (\mathbf{e}(x))^2 \mathrm d x\right].
\end{align*}
Now note that $\int_0^1 (\mathbf{e}(x))^2 \mathrm d x$ has finite expectation because it is smaller than $(\sup \exc)^2,$ which is indeed integrable ($\sup \exc$ has sub-Gaussian tails, see \cite{kennedy76}). So the above quantity is $O(\sigma^3)$.  This finishes the proof of $(ii)$ and $(iii)$.

Finally, since $N_{\sigma}^a$ is integer-valued, we have
\begin{align*}
\E \left[N_{\sigma}^a\mathbf{1}_{\{N_{\sigma}^a\geq 2\}} \right]&=\E[N_{\sigma}^a]-\pr[N_{\sigma}^a=1] \\
                                  &=\frac{\E[2\sigma^{3/2}\mathcal{A}e^{\sigma^{3/2}\mathcal{A}}]}{\E[e^{\sigma^{3/2}\mathcal{A}}]} -\pr[N_{\sigma}^a=1] \\
                                 &= \frac{2\sigma^{3/2}\left(F'(\sigma^{3/2})-F'(-\sigma^{3/2})\right)}{F(\sigma^{3/2})}.
\end{align*}
This proves that $\E[N_{\sigma}^a\mathbf{1}_{\{N_{\sigma}^a\geq 2\}}]=O(\sigma^3)$ as $\sigma \to 0$, but we also want the bound as $\sigma$ tends to infinity. To this end, we write
\[\E[N_{\sigma}^a\mathbf{1}_{\{N_{\sigma}^a\geq 2\}}]\leq 2\sigma^{3/2}\frac{F'(\sigma^{3/2})}{F(\sigma^{3/2})}\]
and simply aim to prove that $F'(x)=O(xF(x))$ as $x\to \infty.$ Quoting \cite[Section 7]{janson2007}, we have
\[F(x)=\sum_{n=0}^{\infty} a_nx^n,\]
where 
\[a_n \underset{n\to\infty}\sim \frac{3\sqrt{2}}{(n-1)!}\left(\frac{n}{12 e}\right)^{n/2}.\]
The desired domination will follow from the fact that $\frac{(n+1)a_{n+1}}{a_{n-1}}$ (the ratio of the coefficients of $x^n$ in $F'(x)$ and $xF(x)\,$) is uniformly bounded for $n\geq 1,$ which is true, since the sequence in fact converges:
\begin{align*}
\frac{(n+1)a_{n+1}}{a_{n-1}} &\sim \frac{n+1}{n(n-1)}\frac{1}{12e}\frac{(n+1)^{(n+1)/2}}{(n-1)^{(n-1)/2}}   \sim \frac{1}{12e} \frac{(n+1)^{n/2}}{(n-1)^{n/2}} \to \frac{1}{12}.
\end{align*}
This completes the proof.
\end{proof}

\subsection{Some properties of the scaling limit}

Let $(\sigma_1,\sigma_2,\ldots)$ be the lengths of the excursions of $B^{\lambda}$, listed in decreasing order. For each $i\geq 1$, let $\mathcal{D}_i$ be an independent copy of $\mathcal{M}_{\sigma_i}$ and let $\mathcal{D}=\bigcup_{i=1}^{\infty} \mathcal{D}_i$. We think of $\mathcal{D}$ as a countable MDM.

\begin{thm} 
\begin{itemize}
\item[(i)] The number of complex connected components of $\mathcal{D}$ has finite expectation.
\item[(ii)] The number of loops of $\mathcal{D}$ is a.s.\ infinite.
\end{itemize}
\end{thm}

\begin{proof}
We start with part $(i).$ For each $i \ge 1$, let $K_i$ be the number of complex components in $\mathcal{D}_i.$ Each complex component contains at least one ancestral identification and so $K_i\leq N_{\sigma_i}^a$. Furthermore, if there is exactly one ancestral identification, there must also be at least one which is non-ancestral in order to obtain a complex component, so that $\pr[K_i=1 \, | \,  \sigma_i]\leq  \pr[N_{\sigma_i}^a=1,N_{\sigma_i}^b\geq 1 \, | \,  \sigma_i]+\pr[N_{\sigma_i}^a\geq 2 \, | \,  \sigma_i]$. Hence, by $(iii)$ and $(iv)$ from Proposition~\ref{prop:Poissonbounds},
\begin{align*}
\E[K_i \, | \, \sigma_i]&= \pr[K_i=1 \, | \, \sigma_i] + \E[K_i\mathbf{1}_{K_i\geq 2} \, | \, \sigma_i] \\
                     &\leq \pr[N_{\sigma_i}^a=1,N_{\sigma_i}^b\geq 1 \, |\, \sigma_i]+\pr[N_{\sigma_i}^a\geq 2 \, | \, \sigma_i] + \E[N_{\sigma_i}^a\mathbf{1}_{\{N_{\sigma_i}^a\geq 2\}}\mid \sigma_i] \\
                     & \leq C\sigma_i^3 
\end{align*}
for some $C>0.$ Thus,
\[
\E\left[\sum_{i=1}^{\infty} K_i\right] \leq C\E\left[\sum_{i=1}^{\infty} \sigma_i^3\right] < \infty. 
\]                

For part $(ii),$ notice first that, since we now know that $\mathcal{D}$ has finitely many complex components a.s., it is sufficient to show that there are infinitely many ancestral identifications, i.e.\ that $\sum_{i=1}^{\infty} N_{\sigma_i}^a=\infty$ a.s. But since $\pr[N_{\sigma_i}^a \geq 1 | (\sigma_j,j\in\N)]$ is asymptotically equivalent to $2c\sigma_i^{3/2}$ by Proposition~\ref{prop:Poissonbounds}, and Proposition~\ref{prop:excursionlengths} gives $\sum_i \sigma_i^{3/2}=\infty$ a.s., the claim follows from an application of the Borel-Cantelli lemma.
\end{proof}

The following property of $\mathcal{D}$ is not surprising, but nonetheless requires proof.
\begin{prop}\label{prop:differentlengths}
The strongly connected components of $\mathcal{D}$ all have different lengths a.s.
\end{prop}
This follows straightforwardly from the following lemma, in which we work on a single tree.
\begin{lem} Let $\sigma>0$. 
\begin{itemize}
\item[$(i)$] For all $x>0$, $\pr[\mathcal{M}_{\sigma} \text{ has a strongly connected component of length } x]=0$.
\item[$(ii)$] $\pr[\mathcal{M}_{\sigma} \text{ has two strongly connected components with equal lengths}]=0$.
\end{itemize}
\end{lem}


\begin{proof} Let $f = 2\tilde{\exc}^{(\sigma)}$ be the excursion function encoding the tree $\T_{\sigma}$ from which $\mathcal{M}_{\sigma}$ is obtained, let the selected leaves be $(x_i,i\in \{1,\ldots,N\}),$ and let $\mathcal{C}^{(\sigma)}_1,\ldots,\mathcal{C}^{(\sigma)}_K$ be the  strongly connected components of $\mathcal{M}_{\sigma}$, listed in the order of appearance of their first elements in the planar ordering of $\T_{\sigma}$. For each $k\in\N,$ on the event where $k\leq K,$ let $E_k=\{i\in\{1,\ldots,N\}:\; x_i\in \mathcal{C}^{(\sigma)}_k\}$ be the set of indices of the leaves implicated in the construction of the $k$th strongly connected component, let $u_k$  be the most recent common ancestor of those leaves. Let $\rho_k = \sup\, \{x\in \llbracket \rho,u_k\rrbracket: x\text{ is a branchpoint and }x\notin \mathcal{C}^{(\sigma)}_k\},$  $\T_k= \bigcup_{i\in E_k} \llbracket \rho_k, x_i\rrbracket,$  and finally let $n_k= \#\{i\in E_k, y_i \in \llbracket\rho_k,u_k\rrbracket\}$, be the number of heads along the line-segment separating $\rho_k$ from $u_k$. Notice then that the length of $\mathcal{C}^{(\sigma)}_k$ is exactly that of $\T_k$, minus the initial part between $\rho_k$ and the first $y_i$ to be encountered. However, since the $y_i$ are chosen uniformly from the length measure, this means that $\llbracket \rho_k,u_k\rrbracket$ is split according to a Dirichlet distribution with $n_k+1$ components. More specifically, we have
\[\mathrm{len}(\T_k)-\mathrm{len}(\mathcal{C}^{(\sigma)}_k)=(f(u_k)-f(\rho_k))\Delta^k_1\]
where, conditionally on $n_k,$ $\Delta^k_1$ is the first component of a vector $\Delta^k=(\Delta_1^k,\ldots,\Delta_{n_k+1}^k)$ which has Dirichlet$(1,\ldots,1)$ distribution. (See Figure~\ref{fig:lengths} for an illustration.) Since Dirichlet distributions have a density, we obtain
\[\pr[\mathrm{len}(\mathcal{C}^{(\sigma)}_k)=x \mid k\leq K, \mathrm{len}(\T_k),f(u_k),f(\rho_k),n_k]=0,\]
and integrating and taking the union over all $k$ gives us $(i).$ 

\begin{figure}
\centering
\includegraphics[scale=0.8]{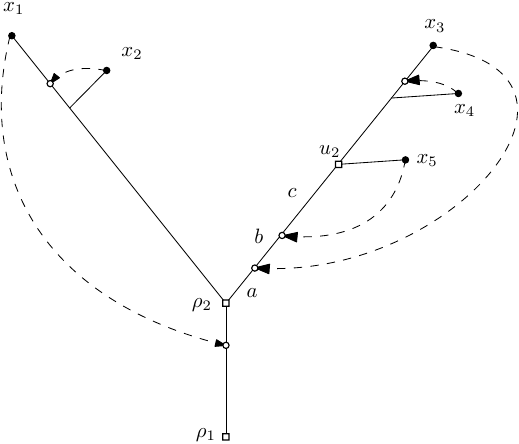}
\caption{For point $(i)$, focusing on the second component, $\T_2$ contains the leaves $x_3,x_4,x_5$, and the length of its initial segment $f(u_2)-f(\rho_2)$ is split by a Dirichlet$(1,1,1)$ into $(a,b,c).$ For $(ii),$ conditioning on $\rho_2$ being in $\T_1,$ then this split is still Dirichlet$(1,1,1).$}
\label{fig:lengths}
\end{figure}

To prove $(ii)$, consider two distinct integers $k$ and $l$. If $k\leq K$ and $l\leq K$, let
\[A_k = \{\rho_k\not\in \mathcal{C}^{(\sigma)}_l\}\]
and
\[A_l = \{\rho_l\not\in \mathcal{C}^{(\sigma)}_k\}.\]
Observe that $\pr[A_k\cup A_l]=1,$ since $\mathcal{C}^{(\sigma)}_k$ and $\mathcal{C}^{(\sigma)}_l$ do not intersect.  Now, on the event $A_l,$ $\T_k$ and $\T_l$ intersect either at point $\rho_k$ or not at all, and we can still write
\[\mathrm{len}(\T_k)-\mathrm{len}(\mathcal{C}^{(\sigma)}_k)=(f(u_k)-f(\rho_k))\Delta^l_1\]
where, conditionally on $\T_l,$ $n_k$ and the event $A_l,$ $\Delta^l_1$ is the first component of a Dirichlet$(1,\ldots,1)$ vector. (Again see Figure~\ref{fig:lengths}.) This means that the length of $\mathcal{C}^{(\sigma)}_k$ has a (conditional) density, and integrating, we get
\[\pr \left[\mathrm{len}(\mathcal{C}^{(\sigma)}_k)=\mathrm{len}(\mathcal{C}^{(\sigma)}_l) \Big| A_l,\; k,l\leq K\right]=0.\]
Symmetrising then yields that
\[\pr\left[\mathrm{len}(\mathcal{C}^{(\sigma)}_k)=\mathrm{len}(\mathcal{C}^{(\sigma)}_l) \Big| k,l\leq K\right]=0,\]
and taking a countable union yields $(ii)$.
\end{proof}

\begin{proof}[Proof of Proposition~\ref{prop:differentlengths}] We label the strongly connected components of $\mathcal{D}$ in such a way that, for $i\in\N,$ those which belong to $\mathcal{D}_i$ are called $\mathcal{C}_{i,1},\ldots,\mathcal{C}_{i,K_i}.$ Consider $\mathcal{C}_{i,k}$ and $\mathcal{C}_{j,l}$ for $i,j,k,l$ in $\N.$ We can assume $i\neq j$ as the case where $i=j$ has already been treated. Conditionally on the excursion lengths $(\sigma_i,i\in \N)$, $\mathcal{C}_{i,k}$ and $\mathcal{C}_{j,l}$ are independent and we have $\pr[\mathrm{len}(\mathcal{C}_{i,k})=x]=0$ for all $x>0.$ Thus we have $\pr[\mathrm{len}(\mathcal{C}_{i,k})=\mathrm{len}(\mathcal{C}_{j,l})\mid \mathrm{len}(\mathcal{C}_{j,l})]=0,$ and integrating to remove the conditioning yields $\pr[\mathrm{len}(\mathcal{C}_{i,k})=\mathrm{len}(\mathcal{C}_{j,l})]=0$.   This completes the proof.
\end{proof}

\section{Convergence of the strongly connected components}\label{sec:mainproof}

For $n\in\N,$ let $p=p(n)$ be such that $p=1/n+\lambda n^{-4/3} + o(n^{-4/3})$ as $n\to\infty.$ Recall that $(C_i(n),i\in\N)$ are the strongly connected components of $\vec{G}(n,p)$, listed in decreasing order of size (with ties broken by using the increasing order of smallest vertex-label), where we treat isolated vertices as copies of the loop of zero length, and additionally append infinitely many copies of the loop of zero length. Let $(\mathcal{C}_i,i\in\N)$ be the strongly connected components of $\mathcal{D},$ listed in decreasing order of length. 

We restate the main theorem.
\medskip

\noindent \textbf{Theorem~\ref{thm:main}.}
\[\left(\frac{C_i(n)}{n^{1/3}},i\in\N \right) \cv  (\mathcal{C}_i,i\in\N)\]
\emph{with respect to the distance $d$ defined by
\[
d(\mathbf{A},\mathbf{B})=\sum_{i=1}^{\infty} d_{\G}(A_i,B_i),
\]
for $\mathbf{A}, \mathbf{B} \in \G^{\N}$.}
\medskip

The aim of this section is to prove this theorem.  We begin by discussing some topological issues related to $d_{\G}$.  We then prove a series of preliminary results, before finally turning to the proof of Theorem~\ref{thm:main}.

\subsection{The relationship between $d_{\G}$ and the Gromov--Hausdorff distance}
Recall from the introduction the definition of metric directed multigraphs (MDMs), and that the distance between two such objects $X=(V,E,r,\ell)$ and $X'=(V',E',r',\ell')$ is defined by
\[
d_{\G}(X,X')= \inf_{(f,g)\in\mathrm{Isom}(X,X')} \ \sup_{e\in E} \ |\ell(e)-\ell'(g(e))|.
\]
Elements of $\vec{\mathcal{G}}$ can also be viewed as metric spaces, by thinking of each edge with positive length as a line segment, identifying vertices joined by edges with length $0$, and forgetting the orientation of the edges. This means that we can also compare them using the Gromov--Hausdorff distance $d_{\mathrm{GH}}$ (see Appendix~\ref{sec:appendix1} for a short introduction). The resulting topology is however weaker, as the following lemma shows.

\begin{lem}\label{lem:distancedomination} For $X\in\G$ and $X'\in\G,$ we have
\[d_{\mathrm{GH}}(X,X')\leq \frac{1}{2}|E| \, d_{\G}(X,X')\]
\end{lem}

\begin{proof} If $X$ and $X'$ do not have the same graph structure, then $d_{\G}(X,X')=\infty$ and the statement holds trivially. If they do have the same graph structure then, up to applying an optimal isomorphism $(f,g)$, we can assume that they have the same vertex and edge sets, i.e.\ $X=(V,E,r,\ell)$ and $X'=(V,E,r,\ell'),$ where the length assignments $\ell$ and $\ell'$ are such that $\sup_{e \in E} |\ell(e)-\ell'(e)|=d_{\G}(X,X')$. We let $\phi$ be the natural bijection from $X$ to $X'$ when viewed as metric spaces, which acts identically on $V$ and follows the edges ``linearly''.  Viewing $\phi$ as a correspondence (see Appendix~\ref{sec:appendix2}), its distortion can be bounded above by 
\[
\sum_{e \in E} |\ell(e)-\ell'(e)| \leq |E| \, \sup_{e \in E}\, |\ell(e)-\ell'(e)|= |E| \, d_{\G}(X,X'). \qedhere 
\]
\end{proof}

In the case of trees, it is possible to recover a convergence for $d_{\G}$ from a \emph{pointed} Gromov--Hausdorff convergence (see Appendix~\ref{sec:appendix1} for a definition).  (Variants of these ideas have been used in several places in the literature, and we do not claim that the following proposition is particularly novel.  We have not, however, found a convenient reference.)

\begin{prop}\label{prop:cvtreesG} Fix $k\in\N.$ For $n\in\N$, let $(\mathcal{T}_n,n\in\N)$ (resp.\ $\mathcal{T}$) be $\R$-trees with roots $\rho_n$ (resp.\ $\rho$) and $k$ selected distinct leaves $(x_{i,n},1\leq i \leq k)$ (resp.\ $(x_i,i\leq k)$). Then let
\[
\overline{\mathcal{T}}_n= \bigcup_{i=1}^k \llbracket \rho_n,x_{i,n}\rrbracket
\]
be the subtree spanned by the $k$ selected leaves and the root (and define $\overline{\mathcal{T}}$ similarly). View it as an element of $\G$ by taking as vertices the root, the leaves, and all the branch points, orienting each edge away from $\rho_n$ (resp. $\rho$) and giving each edge the length of its corresponding metric path. 

Suppose that $(\mathcal{T}_n,\rho_n, x_{1,n}, \ldots, x_{k,n})$ converges to $(\mathcal{T},\rho,x_1, \ldots x_k)$ for the $(k+1)$-pointed Gromov--Hausdorff topology, and that $\overline{\mathcal{T}}$ is binary. Then $\overline{\mathcal{T}}_n$ converges to $\overline{\mathcal{T}}$ for $d_{\G}.$ Specifically, $\overline{\T}_n$ and $\overline{\T}$ are seen as elements of $\G$ by taking as vertices their roots, leaves, and branchpoints and directing edges away from the root. The map which sends $\rho_n$ to $\rho$ and $x_{i,n}$ to $x_i$ for each $1 \leq i \leq k$ then extends uniquely to a graph isomorphism, under which the length of each edge in $\overline{\mathcal{T}}_n$ converges to that of the corresponding edge in $\overline{\mathcal{T}}.$
\end{prop}

\begin{proof}
For notational convenience, we let $x_{0,n}=\rho_n$ and $x_0=\rho$ in the following proof. The pointed Gromov--Hausdorff convergence which is assumed to hold in the statement of the proposition implies that, for any $i$ and $j$ in $\{0,\ldots,k\},$ we have that $d(x_{i,n},x_{j,n})$ converges to $d(x_i,x_j)$ as $n\to\infty.$ Indeed, by taking a suitable embedding for each $n$, we can see that 
\[
d(x_i,x_j)-2\eta\leq d(x_{i,n},x_{j,n})\leq  d(x_i,x_j)+2\eta
\]
for any $\eta> d_{\mathrm{GH}}((\mathcal{T}_n,\rho_n, x_{1,n}, \ldots, x_{k,n}),(\mathcal{T},\rho,x_1, \ldots x_k))$.
Next, using the formula $d(\rho,x\wedge y)=\frac{1}{2}(d(\rho,x)+d(\rho,y)-d(x,y))$ (which is valid in any $\R$-tree), we see that these convergences extend to distances of the form $d(\rho_n,x_{i,n}\wedge x_{j,n})$ and $d(x_{i,n},x_{i,n}\wedge x_{j,n}).$

We will now prove the proposition by induction, adding the leaves one by one. Specifically, for $l\leq k,$ let $\T_n^l$ and $\T^l$ be the subtrees of $\T_n$ and $\T$ spanned by the root and the first $l$ selected leaves.  We prove that the map mentioned at the end of the proposition, restricted to $\T_n^l,$ gives convergence in $d_{\G}$, by induction on $l$.
 
The base case $l=1$ is simple, since all of the trees we consider are just single line segments, and $d(\rho_n,x_{1,n})$ converges to $d(\rho,x_1)$ as mentioned above. So let us now focus on the induction step, and assume the proposition at rank $l$ for some $l\leq k-1.$ Observe then that we only need to prove two things: first, that the new branchpoint of $\T_n^{l+1}$ is, for $n$ large enough, added to the same edge as the new branchpoint of $\T^{l+1}$ and, second, that the height of this branchpoint converges. Indeed, the first point will give the desired graph isomorphism, while the second, combined with the convergence of $d(\rho_n,x_{l+1,n}),$ will give convergence of the lengths of all the edges.

Let $\llbracket y,z\rrbracket$ be the edge of $\T^l$ which contains the new branchpoint of $\T^{l+1}$, with $y$ being an ancestor of $z$. Since $\T$ is binary, this new branchpoint is equal to neither $y$ nor $z$; let us call it $b$. This implies that $|d(x_{l+1},y)-d(x_{l+1},z)| = |d(y,b) - d(z,b)| < d(y,b) + d(b,z) = d(y,z)$. Now, letting $y_n$ and $z_n$ be the points in $\T_n^l$ corresponding to $y$ and $z$ respectively  through the graph isomorphism, and using that distances in $\T_n^{l}$ converge to those in $\T^l$, we have that $|d(x_{l+1,n},y_n)-d(x_{l+1,n},z_n)|<d(y_n,z_n)$ for $n$ sufficiently large.  This implies that the new branchpoint lies between $y_n$ and $z_n$.  (Indeed, if the new branchpoint did not lie in $\llbracket y_n,z_n \rrbracket$, then we would have that the path from the new branchpoint to one of $y_n$ and $z_n$ would need to traverse the whole edge $\llbracket y_n,z_n \rrbracket$ and then we would have $|d(x_{l+1,n},y_n)-d(x_{l+1,n},z_n)| = d(y_n,z_n)$.)  Finally, by observing that this branchpoint can consistently be written as $x_{i,n}\wedge x_{l+1,n},$ with $i$ chosen such that $x_i$ is a descendant of $y$, our initial remark on convergence of distances concludes the induction and the proof.
\end{proof}

Note that Proposition \ref{prop:cvtreesG} fails if $\bar{\T}$ is not binary. Indeed, it would then be possible for $\bar{\T}_n$ to never have the same graph structure, thus preventing the convergence for $d_{\G}.$

\begin{prop}\label{prop:graphtocc} If the connected components of an MDM $X$ all have different total lengths, and $(X_n,n\in \N)$ is a sequence which converges to $X$ for $d_{\G}$, then the strongly connected components of $X_n$, listed in decreasing order of length and seen as elements of $\vec{\mathcal{G}}$, converge to those of $X$.
\end{prop}

\begin{proof}
Writing $X=(V,E,r,\ell),$ let $(C_1,\ldots,C_k)$ be the strongly connected components of $X,$ ordered by decreasing length. For $n\in\N$ large enough, one may assume we have $X_n=(V,E,r, \ell_n),$ where $\ell_n(C_i) \to \ell(C_i)$ as $n \to \infty$ for all $i$.  In particular, for $n$ large enough, $\ell_n(C_i)$ is strictly decreasing in $i$, and so $(C_1,\ldots,C_k)$ is also the ordered sequence of the strongly connected components of $X_n,$ which completes the proof.
\end{proof}

\subsection{The components originating from a single tree} \label{sec:onetree}
The first part of the proof will consist in proving the convergence of the components originating from a single tree.  For $m\in\N$, we take a plane tree $T_m$ which has the distribution of a tree component of $\mathcal{F}_{\vec{G}(n,p)}$ conditioned to have size $m$. We are interested in $m \sim \sigma n^{2/3}$ so that, in particular, we have $mp^{2/3}\to \sigma$ as $m\to\infty$. From \cite{A-BBG12}, up to an unimportant relabelling of the vertices, $T_m$ has the same distribution as a uniform random labelled tree on $[m]$, biased by $(1-p)^{-a(T_m)},$ where $a(T_m)$ is the number of permitted edges in $T_m$. We give this tree a planar embedding by rooting at the vertex labelled 1 (we also refer to this root as $\rho_m$) and then simply using the increasing order on the labels of the children of any vertex.  Let $H^m: \{0,\ldots,m-1\} \to \Z_+$ be the height function of $T_m$, such that $H^m(i)$ is the height of the $i$-th vertex in the planar order, starting with $H^m(0)=0.$ We recall that $\|T_m\| = \max_{0 \leq i \leq m-1} H^m(i)$ is the height of the tree $T_m$.  Theorem 15 of \cite{A-BBG12} states that
\begin{equation}\label{eq:excursionconvergence}
\big((m/\sigma)^{-1/2}H^m(\lfloor(m/\sigma)t\rfloor),0\leq t\leq \sigma\big) \cv  (2\tilde{\exc}^{(\sigma)}(t),0\leq t\leq \sigma)
\end{equation}
uniformly as $m \to \infty$.  By Lemma~\ref{lem:pointedGHexcursion} this has the straightforward consequence that
\begin{equation} \label{eqn:treeconv}
\left(\left(\frac{\sigma}{m}\right)^{1/2} T_m, \rho_m \right) \cv (\mathcal{T}_{\sigma}, \rho)
\end{equation}
as $m \to \infty$, for the 1-pointed Gromov--Hausdorff distance (see (\ref{eqn:kpointedGH}) for a definition).

As in Proposition~\ref{prop:coupling}, we include each of the $\binom{m}{2}$ possible back edges and $a(T_m)$ possible surplus edges independently with probability $p$, and let $X_m$ be the resulting directed graph. The aim of this section is to show that the rescaled strongly connected components of $X_m$ converge in distribution to those of $\mathcal{M}_{\sigma}.$ In order to do this, we will use the structure of back edges outlined in Section~\ref{sec:discretebackedges}. Specifically, let $\big((x_{i,m},y_{i,m}),m\in\N,i\leq N_m)$ be the back edges obtained with this procedure, and let $X^*_m$ be the subgraph of $X_m$ obtained by removing (a) all surplus edges and (b) all back edges which are not of the above form. We will first show that the strongly connected components of $X^*_m$ converge in distribution to those of $\mathcal{M}_{\sigma}$, and then that $X_m$ and $X^*_m$ have the same strongly connected components with high probability.  In particular, we show that the surplus edges with high probability do not play any role in creating the strongly connected components.

\subsubsection{Convergence of the marked points}
Our next step is to improve the convergence of the rooted tree $\left(\left(\frac{\sigma}{m}\right)^{1/2}T_m, \rho_m\right)$ to include the marked points $(x_{i,m})$.  Since the number of marked points is random, we use a pointed Gromov--Hausdorff distance $d_{\mathrm{GH}}^*$ which allows for a variable number of marks (see (\ref{eqn:pointedGH}) for a formal definition).

Recall the relevant notation for the limit object from Section~\ref{sec:continuousbackedges}.  In particular, we write $f=2\tilde{\exc}^{(\sigma)}$, we let $s_i, 1 \le i \le N$ be the points of the Poisson process and we let $x_i = p_f(s_i)$ be their projections onto the tree $\mathcal{T}_{\sigma}$ encoded by $f$.

\begin{prop}\label{prop:cvtreewithtails}
We have
\[\left(\left(\frac{\sigma}{m}\right)^{1/2}T_m, \rho_m, \big(x_{i,m},i\leq N_m \big)\right) \cv \Big(\T_{\sigma},\rho,\big(x_i,i\leq N \big)\Big)\]
as $m \to \infty$, for the topology generated by $d_{\mathrm{GH}}^*$.
\end{prop}


%
\begin{proof} By Skorokhod's representation theorem, there exists a probability space on which the convergence (\ref{eq:excursionconvergence}) occurs almost surely, which entails an almost sure convergence in (\ref{eqn:treeconv}) also.  We need to add the other marked points to the convergence. Let $k_{i,m}$ be the index of $x_{i,m}$ in the planar ordering of $T_m$. We will show by induction on $i$ that $\pr[N_m\geq i] \to \pr[N\geq i]$ and that, conditionally on $N_m\geq i$, the rescaled index $(\frac{m}{\sigma})^{-1} k_{i,m}$ converges in distribution to $s_i$.  We can then use Lemma~\ref{lem:pointedGHexcursion} to transfer this convergence of the rescaled indices $(\frac{m}{\sigma})^{-1} k_{i,m}$ to that of the $x_{i,m}$.  Together, these elements will suffice to give the claimed convergence in the sense of $d_{\mathrm{GH}}^*$.

We start with $i=1.$ Since the number of ancestral back edges originating at the $k$-th point of $T_m$ has distribution $\mathrm{Bin}(H^m(k),p),$ and $pH^m(\lfloor xm \rfloor)\sim (\frac{m}{\sigma})^{-1}f(x),$ it is straightforward to see that $\pr[N_m\geq 1]\to \pr[N\geq 1]$ and that, conditionally on $N_m\geq 1$, $(\frac{m}{\sigma})^{-1} k_{1,m}$ converges in distribution to the first point of a Poisson point process with intensity $f(x)\mathrm d x$, also conditioned to have at least one point, which is precisely $s_1$. (See the proof of Lemma 19 of \cite{A-BBG12} for a more detailed version of an essentially identical argument.) 

The induction step uses the same idea. Assume that the claimed convergence holds up to rank $i$.  Recall that $\T_i=\cup_{j=1}^i \llbracket \rho,x_i\rrbracket$.  The index $k_{i+1,m}$ is found by giving to each $k\geq k_{i,m}$ a $\mathrm{Bin}(l_i(k),p)$ number of marks, where $l_i(k)$ is the number of possible heads of a back edge originating at the $k$-th point in the planar ordering of $T_m$. If we let $T_m(i)$ be the subtree of $T_m$ spanned by the root and $x_{1,m},\ldots,x_{i,m}$, then $l_i(k)=|T_m(i)|+H^m(k)-\min \{H^m(l): l\in\{k_{i,m},\ldots,k\}\}.$  By the induction hypothesis and Proposition~\ref{prop:cvtreesG}, we have that $T_m(i)$ converges to $\T_i$ for $d_{\G}$ and so, in particular, $|T_m(i)| \sim \left(\frac{m}{\sigma}\right)^{1/2} \mathrm{len}(\T_i)$.  So we have that the instantaneous rate at which points occur satisfies
\[
p l_i(\lfloor x m \rfloor) \sim \left( \frac{m}{\sigma} \right)^{-1} \left(\mathrm{len}(\mathcal{T}_i) + f(x) - \hat{f}(s_i,x) \right).
\]
It follows that $\pr[N_m \ge i+1|N_m \ge i] \to \pr[N \ge i+1|N \ge i]$ and that, conditionally on $N_m \ge i+1$, we have $(\frac{m}{\sigma})^{-1} k_{i+1,m} \cv s_{i+1}$.  Hence, by induction we have that
\[
N_m \cv N
\]
and
\[
(s_{i,m}, 1 \le i \le N_m) \cv (s_i, 1 \le i \le N)
\]
as $m \to \infty$.  The proposition then follows by applying Lemma~\ref{lem:pointedGHexcursion}.
\end{proof}

\begin{prop}\label{prop:cvcolouredtree} Let $T^{\mathrm{mk}}_m :=\bigcup_{i=1}^{N_m} \llbracket \rho_m,x_{i,m} \rrbracket$ and $\T_{\sigma}^{\mathrm{mk}} :=\bigcup_{i=1}^{N} \mathcal{T}_i$ be the marked subtrees of $T_m$ and $\T_{\sigma}$ respectively. Then
\[
\left(\left(\frac{\sigma}{m}\right)^{1/2}T_m^{\mathrm{mk}}, \rho_m, \big((x_{i,m},y_{i,m}),i\leq N_m \big)\right) \cv \Big(\T_{\sigma}^{\mathrm{mk}}, \rho,\big((x_i,y_i),i\leq N \big)\Big),
\]
as $m \to \infty$ for the topology generated by $d_{\mathrm{GH}}^*$ and, moreover,
\[
\left(\frac{\sigma}{m}\right)^{1/2}|T^{\mathrm{mk}}_m| \cv \mathrm{len}(\T_{\sigma}^{\mathrm{mk}}).
\]
\end{prop}
\begin{proof}
The pointed Gromov--Hausdorff convergence of Proposition~\ref{prop:cvtreewithtails} can be restricted to the marked subtrees by using the same embeddings. Using Skorokhod's representation theorem, we may thus work on a probability space where 
\[
\left(\left(\frac{\sigma}{m}\right)^{1/2}T_m^{\mathrm{mk}}, \rho_m, (x_{i,m}, i \le N_m) \right) \underset{m\to\infty}{\longrightarrow}\Big(\T_{\sigma}^{\mathrm{mk}},\rho, (x_i, i \le N) \Big)
\]
almost surely for $d_{\mathrm{GH}}^*$. (In particular, almost surely $N_m = N$ for all $m$ sufficiently large.)  We will add the $y$ terms to this convergence by using their distributions and a correspondence. The basic idea is that $y_{i,m}$ has the uniform distribution on the finite set $T_m(i)\setminus \{x_{i,m}\}$, which converges to the normalised length measure on $\mathcal{T}_i$.

By Proposition~\ref{prop:cvtreesG}, $T^{\mathrm{mk}}_m$ has the same graph structure as $\T_{\sigma}^{\mathrm{mk}}$ for $m$ sufficiently large. Let us make this more precise by denoting by $\mathsf{T}^{\mathrm{mk}}$ the graph which has as vertices the root, branchpoints and leaves of $\T_{\sigma}^{\mathrm{mk}},$ and $\mathsf{E}^{\mathrm{mk}}$ its set of edges. Then $\T_{\sigma}^{\mathrm{mk}}$ can be seen as a MDM by giving to each edge $e \in \mathsf{E}^{\mathrm{mk}}$ its length $\ell(e)$, while $T_m^{\mathrm{mk}}$ is the graph obtained by splitting each edge $e \in \mathsf{E}^{\mathrm{mk}}$ into a path of length $k_m(e)\in\N$. Note that we have (again by Proposition~\ref{prop:cvtreesG}) that
\[
\left(\frac{\sigma}{m}\right)^{1/2}k_m(e) \to \ell(e).
\] 
This allows us to build a function $\phi_m$ which naturally injects $T_m^{\mathrm{mk}}$ into $\T_{\sigma}^{\mathrm{mk}}.$  Specifically, the vertices of $T_m^{\mathrm{mk}}$ which also belong to $\mathsf{T}^{\mathrm{mk}}$ are mapped so as to preserve the structure, and the vertices of degree 2 in $T_m^{\mathrm{mk}}$ which subdivide the edge $e$ are mapped ``linearly", dividing the corresponding edge of $\T_{\sigma}^{\mathrm{mk}}$ into $k_m(e)$ segments of equal length. The inverse of this injection then may be naturally extended to a projection $\psi_m$ from $\T_{\sigma}^{\mathrm{mk}}$ onto $T_{m}^{\mathrm{mk}}$ by letting $\psi_m(x)$ be the most recent ancestor of $x$ which belongs to $\phi_m(T_m^{\mathrm{mk}})$. This allows us to define a correspondence $\mathcal{R}$ between $T_m^{\mathrm{mk}}$ and $\T_{\sigma}^{\mathrm{mk}}$ by letting $\psi_m(x) \mathcal{R}\, x$ for all $x\in \T_{\sigma}^{\mathrm{mk}}$. 

The distortion of $\mathcal{R}$ is then
\[\mathrm{dis}\,\mathcal{R}= \underset{x,y\in \T_{\sigma}^{\mathrm{mk}}}\sup \left|d(x,y)-\left(\frac{\sigma}{m}\right)^{1/2}d(\psi_m(x),\psi_m(y))\right|.\]
Now, for any $x, y \in \T_{\sigma}^{\mathrm{mk}},$  
\begin{align*}
& \left| d(x,y)-\left(\frac{\sigma}{m}\right)^{1/2}d(\psi_m(x),\psi_m(y)) \right| \\
& \quad \le d(x,\phi_m(\psi_m(x))+d(y,\phi_m(\psi_m(y))  + \left|d(\phi_m(\psi_m(x)),\phi_m(\psi_m(y)))-\left(\frac{\sigma}{m}\right)^{1/2}d(\psi_m(x),\psi_m(y))\right| \\
& \quad \le 2\,\underset{z\in\T_{\sigma}^{\mathrm{mk}}}\sup d(z,\phi_m(\psi_m(z)) + \underset{u,v\in T_m^{\mathrm{mk}}}\sup \left|\left(\frac{\sigma}{m}\right)^{1/2}d(u,v)-d(\phi_m(u),\phi_m(v))\right| \\
& \quad \le 2\,\underset{e\in\mathsf{E}^{\mathrm{mk}}}\sup \,\frac{\ell(e)}{k_m(e)} + \sum_{e\in\mathsf{E}^{\mathrm{mk}}} \left|\left(\frac{\sigma}{m}\right)^{1/2}k_m(e)-\ell(e)\right|.
\end{align*}
 This upper bound is easily seen to tend to $0$ as $m \to \infty$, and so $\mathrm{dis}\,\mathcal{R} \to 0$ also.

For $i\leq N,$ let $\nu_i$ be the normalised length measure on $\T_i$, and let $\nu_{i,m}$ be the uniform measure on $T_m(i)\setminus\{x_i\}.$ We aim to apply Lemma~\ref{lem:GHPtopointed}. Notice that  $\nu_i(\phi_m^{-1})$ is the probability measure on $T_m(i)$ which gives a weight proportional to $\mathrm{deg}(v)$ to any vertex $v$. We want to construct a coupling between this and $\nu_{i,m}$. Let $B_i(m)$ be the set containing the branchpoints of $T_m(i)$ as well as $x_{i,m}$.  Then the measures $\nu_{i,m}$ and $\nu_i$ are both equal to the uniform measure when conditioned on $T_m(i)\setminus B_i(m)$. Since the cardinality of $B_i(m)$ does not change as $m$ increases, we obtain that both $\nu_{i,m}(B_i(m))$ and $\nu_i(\phi_m^{-1}(B_i(m))$ tend to $0$. Then for any $\veps>0$, if $m$ is large enough, there exists a coupling $(U_m,V_m)$ of $\nu_{i,m}$ and $\nu_i(\phi_m^{-1})$ such that $U_m=V_m$ with probability at least $1-\veps.$ By construction, we can write $V_m=\psi_m(W_m)$, and $(U_m,W_m)$ is a coupling of $\nu_{i,m}$ and $\nu_i$ such that $U_m\, \mathcal{R} \,W_m$ with probability at least $1-\veps.$

It follows that 
\[
\left(\left(\frac{\sigma}{m}\right)^{1/2}T_m^{\mathrm{mk}}, \rho_m,(x_{i,m},i\leq N_m),(\nu_{i,m},i\leq N_m)\right) \cv \left(\T_{\sigma}^{\mathrm{mk}}, \rho,(x_{i},i\leq N),(\nu_{i},i\leq N)\right)
\]
for the topology generated by $d^*_{\mathrm{GHP}}$. The proposition then follows by applying Lemma~\ref{lem:GHPtopointed}.
\end{proof}

\subsubsection{Convergence of the marked graph}
Let $X^*_m=T^{\mathrm{mk}}_m$ along with all back edges $(x_{i,m},y_{i,m})$ for $i\leq N_m$, and recall that $\mathcal{M}_{\sigma}=\T_{\sigma}^{\mathrm{mk}} / \sim,$ where $\sim$ is the equivalence relation which identifies $x_i$ with $y_i$ for $i\leq N.$ We view these objects as elements of $\vec{\mathcal{G}},$ in a way which will fit the metric on $\vec{\mathcal{G}}.$ Specifically, we take the vertex set of $X_m^*$ to consist of $\rho$, the heads $y_{i,m}$ of the back edges for $i \le N_m$, and the branch points $x_{i,m} \wedge x_{j,m}$ for $i \neq j \le N_m$.  We take the vertices of $\mathcal{M}_{\sigma}$ to be $\rho$, $y_i$ for $i \le N$ (note that post-identification we have $x_i=y_i$), and the branch points $x_i \wedge x_j$ for $i \neq j \le N$. Because the Brownian continuum random tree is almost surely binary and the law of $\mathcal{T}_{\sigma}$ is absolutely continuous with respect to that of the Brownian continuum random tree, $\mathcal{T}_{\sigma}^{\mathrm{mk}}$ is also binary almost surely.  It follows that $\mathcal{M}_{\sigma}$ has $2N$ vertices and, as we will see, the same must also be true for $X_m^*$ for sufficiently large $m$.

\begin{prop}
$(\frac{\sigma}{m})^{1/2}X_m^* \cv \mathcal{M}_{\sigma}$ in $\G.$
\end{prop}

\begin{proof}
Using Skorokhod's representation theorem, we may assume that the convergence of Proposition~\ref{prop:cvcolouredtree} holds almost surely. Recall that, by Proposition~\ref{prop:cvtreesG}, $((\frac{\sigma}{m})^{1/2}T^{\mathrm{mk}}_m, \rho, (x_{i,m},i \le N_m))$ converges in $\G$ (taking the root, $(x_{i,m})$ and branch points as vertices). In particular the elements of this sequence have the same underlying graph structure for all $m$ large enough.

For $m$ large enough, no $x_{i,m}$ is an ancestor of a $x_{j,m}$ or $y_{j,m},$ so the graph structure of $X_m^*$ can be obtained from that of $T_m^{\mathrm{mk}}$ by removing $x_{i,m}$ and instead connecting the edge ending in $x_{i,m}$ back into $y_{i,m}$, for each $i$. Since $y_{i,m}$ converges to $y_i$ in the Gromov--Hausdorff sense it will, in particular, always be on the same edge of $T_m^{\mathrm{mk}}$ for $m$ sufficiently large. Thus the combinatorial structure is constant for $m$ large, and the same as that of $\mathcal{M}_{\sigma}.$

Once we know the combinatorial structure, the lengths of all the edges then also converge since they can be expressed in terms of the distances between the root, the $(x_{i,m})$ and the $(y_{i,m}).$

Using also Propositions~\ref{prop:graphtocc} and~\ref{prop:differentlengths}, we then obtain that the connected components of $\left(\frac{\sigma}{m}\right)^{1/2}X_m^*,$ listed in decreasing order of size, converge in the sense of $d_{\G}$ to those of $\mathcal{M}_{\sigma},$ listed in decreasing order of length.
\end{proof}

\subsubsection{Surplus edges do not contribute}
As mentioned earlier, we now want to prove that the surplus edges contribute to the strongly connected components of $X_m$ with vanishingly small probability. Specifically, we aim to prove the following proposition.
\begin{prop}\label{prop:surplus} $\pr\left[X_m\text{ and } X^*_m \text{ have different strongly connected components}\right] \to 0$ as $m \to \infty$.
\end{prop}

Let $R(m)$ be the number of surplus edges in $X_m$.  For $1 \le i \le R(m)$, let $\alpha_{i,m}$ and $\beta_{i,m}$ be the tail and head respectively of the $i$-th surplus edge in increasing planar order of their tails. Let $W_i(m)$ be the number of vertices descending from $\beta_{i,m}$ in $T_m$. Proposition~\ref{prop:surplus} will follow if we can establish that the family $\left(\sum_{i=1}^{R(m)}W_i(m),m\in\N \right)$ is tight, namely if
\begin{equation}\label{eq:tightness}
\lim_{K \to \infty} \limsup_{m \to \infty} \; \pr\left[\sum_{i=1}^{R(m)} W_i(m) >  K \right]=0.
\end{equation} 
Indeed, for a strongly connected component of $X_m$ to feature a surplus edge, we need at least one back edge to originate from a descendant of some $\beta_{i,m}$ (since any surplus edge in a strongly connected component is part of a cycle and must thus lead to a back edge after following tree edges or further surplus edges). By Proposition~\ref{prop:coupling}, conditionally on $\sum_{i=1}^{R(m)} W_i(m) \leq K$, the probability of this event is smaller than the probability that a $\mathrm{Bin}(mK,p)$ variable is non-zero. Assuming (\ref{eq:tightness}) and fixing $\veps>0$, we may find a $K$ sufficiently large that 
\[ 
\limsup_{m \to \infty} \;\pr\left[\sum_{i=1}^{R(m)} W_i(m) > K\right]\leq \veps/3, \] and $m$ large enough such that $\pr\left[\sum_{i=1}^{R(m)} W_i(m) \geq K\right]\leq \veps/2$ and $1-(1-p)^{mK}\leq \veps/2$ (recall that $p\sim \sigma^{3/2}m^{-3/2}$ ). Then
\begin{align*}
& \pr\left[X_m\text{ and } X^*_m \text{ have different strongly connected components}\right] \\
& \qquad \leq\pr\left[\sum_{i=1}^{R(m)} W_i(m) \geq K\right] + 1-(1-p)^{mK}  \leq\frac{\veps}{2}+\frac{\veps}{2} = \veps.
\end{align*}

As we have already mentioned, it is shown in \cite{A-BBG12}  that $T_m$ is a biased version of the uniform labelled tree $\mathsf{T_m}$ on $[m]$ (with a canonical planar embedding): for non-negative measurable test functions $f$,
\begin{equation}\label{eq:defbias}
\E[f(T_m)]=\frac{\E[(1-p)^{-a(\mathsf{T}_m)}f(\mathsf{T}_m)]}{\E[(1-p)^{-a(\mathsf{T}_m)}]}.
\end{equation}
We recall that $a(T)$ denotes the the number of surplus edges permitted by the planar structure of a tree $T$, called its \emph{area} in \cite{A-BBG12}. We know from Theorem 12 and Lemma 14 of \cite{A-BBG12} that 
\[
(1-p)^{-a(\mathsf{T}_m)} \cv e^{\int_0^{\sigma}\exc^{(\sigma)}(t)\mathrm dt},
\]
and that the sequence on the left-hand side is bounded in $L^2$. We will prove (\ref{eq:tightness}) by first showing the analogous statement for $\mathsf{T}_m$ (this is Lemma~\ref{lem:notilttightness} below) and then using the measure change. We need the following lemma, which makes use of Kesten's tree, that is the tree $\widehat{\mathsf{T}}$ consisting of a copy of $\Z_{+}$ (the \emph{spine}), at each point of which we graft an independent Galton--Watson tree with Poisson(1) offspring distribution.  We root the resulting infinite tree at 0.  (This is the local weak limit of $\mathsf{T}_m$ \cite{Grimmett}.)

\begin{lem} \label{lem:gwsubtree} 
\begin{itemize}
\item[(i)] Let $\mathsf{Y}(m)$ be the number of vertices of $\mathsf{T}_m$ which lie outside the largest subtree descending from a child of the root. Then 
\[\mathsf{Y}(m) \underset{m\to\infty}{\cv } \mathsf{Y},\]
where $\mathsf{Y}$ is the number of vertices of $\widehat{\mathsf{T}}$ which have no ancestors on the spine apart from the root.

\item[(ii)] Write $(v_i,i\in\{1,\ldots,m\})$ for the vertices of $\mathsf{T}_m$ in planar order. For $v\in\mathsf{T}_m,$ let $\mathsf{Z}_v(m)$  be the number of vertices in the subtree rooted at $v$.  Let $\mathsf{Y}_v(m)$ be the number of such vertices which lie outside the largest of the subtrees rooted at a child of $v.$ Then the $(\mathsf{Y}_{v_i}(m),m\in\N,i\leq m)$ are tight:
\[
\lim_{M \to \infty} \limsup_{m \to \infty} \sup_{1 \le i \le m}  \pr[\mathsf{Y}_{v_i}(m)>M]=0.
\]
\end{itemize}
\end{lem}

\begin{proof} Let $\mathsf{T}_1(m),\ldots,\mathsf{T}_D(m)$ be the subtrees of $\mathsf{T}_m$ rooted at its first generation, with $D$ being the degree of the root, listed in decreasing order of size, and define $\widehat{\mathsf{T}}_i$ similarly (noting that $\widehat{\mathsf{T}}_1$ is infinite). It is well-known that $\mathsf{T}_m$ is a Galton-Watson tree with Poisson(1) offspring distribution, conditioned to have $m$ vertices and assigned a uniformly random labelling from $[m]$. Knowing this, it is shown within the proof of Proposition 5.2 in \cite{pagnard2017} that $(|\mathsf{T}_2(m)|,\ldots |\mathsf{T}_D(m)|,0,\ldots)$ converges in distribution to $(|\widehat{\mathsf{T}}_2|,\ldots,|\widehat{\mathsf{T}}_D|,0,\ldots).$ Since the sum of the latter is a.s.\ finite, this also implies convergence of the sum.

Part $(ii)$ follows from the fact that, for all $i$ and $m$, the conditional distribution of $\mathsf{Y}_{v_i}(m)$ given $\mathsf{Z}_{v_i}(m)$ is the same as that of $\mathsf{Y}(\mathsf{Z}_{v_i}(m)).$ (This is an aspect of the \emph{Markov branching property} of conditioned Galton-Watson trees, see \cite{HM12}.) Hence, we can write, for $M>0$
\begin{align*}
\pr[\mathsf{Y}_{v_i}(m)>M]&= \mathbb{E}[\pr[\mathsf{Y}(Z_{v_i}))>M]\mid Z_{v_i}] \\
						  &\leq \sup_{k\in\N} \pr[\mathsf{Y}(k)>M]
\end{align*} 
Since the distributions of $(\mathsf{Y}(k),k\in\N)$ form a tight sequence, the above upper bound tends to $0$ as $M$ tends to infinity.
\end{proof}

Now add to the tree $\mathsf{T}_m$ each of the $a(\mathsf{T}_m)$ permitted surplus edges independently with probability $p$.  Conditionally on $a(\mathsf{T}_m)$ this yields a $\mathrm{Bin}(a(\mathsf{T}_m),p)$ number of surplus edges, for which we write $\mathsf{R}(m)$. Write the tails and heads of these surplus edges as $\mathsf{a}_{i,m}$ and $\mathsf{b}_{i,m}$ respectively, listed in increasing planar order of $\mathsf{a}_{i,m}$, for $i\leq \mathsf{R}(m)$.  We also write $\mathsf{b}_{i,m}^-$ for the parent of $\mathsf{b}_{i,m}$ in $\mathsf{T}_m$.  Let $\mathsf{W}_i(m)$ be the number of descendants of $\mathsf{b}_{i,m}$. The following lemma is a version of (\ref{eq:tightness}) for $\mathsf{T}_m$.

\begin{lem}\label{lem:notilttightness}
\[
\lim_{K \to \infty} \limsup_{m \to \infty} \; \pr\left[\sum_{i=1}^{\mathsf{R}(m)} \mathsf{W}_i(m) > K\right]=0.\]
\end{lem} 
\begin{proof} Fix $\veps > 0$. In Lemma 19 of \cite{A-BBG12}, it is proved that $R(m)$ converges in distribution as $m \to \infty$. An identical argument shows that $\mathsf{R}(m)$ converges in distribution as $m \to \infty$ and, in particular, is tight. Therefore, there exists $I>0$ such that $\pr[\mathsf{R}(m)>I]<\frac{\veps}{2}$ for all $m.$  Moreover,
\begin{align*}
\pr \left[\sum_{i=1}^{\mathsf{R}(m)} \mathsf{W}_i(m)>K \right] &\leq \pr \left[\mathsf{R}(m)>I \right] + \pr\left[\mathsf{R}(m)\leq I , \sum_{i=1}^{\mathsf{R}(m)} \mathsf{W}_i(m)>K \right] \\
                              &\leq \frac{\veps}{2} + \pr \left[\sum_{i=1}^{\mathsf{R}(m) \wedge I} \mathsf{W}_i(m)>K \right].
\end{align*}
We then split the event where $\sum_{i=1}^{\mathsf{R}(m) \wedge I} \mathsf{W}_i(m)>K$ in two: either, for all $i\leq \mathsf{R}(m) \wedge I$, the vertex $\mathsf{a}_{i,m}$ lies in the largest of the subtrees rooted at the children of $\mathsf{b}_{i,m}^-$, in which case we also have $\sum_{i=1}^{\mathsf{R}(m) \wedge I} \mathsf{Y}_{\mathsf{b}^-_{i,m}}(m)>K$, or there exists $i$ for which $\mathsf{a}_{i,m}$ is \emph{not} in this largest subtree, which then implies, in particular, that $\mathsf{Y}_{\mathsf{b}^-_{i,m}}(m)\geq d(\mathsf{a}_{i,m},\mathsf{b}_{i,m}).$ This leads to
\begin{align*}
& \pr\left[\sum_{i=1}^{\mathsf{R}(m)} \mathsf{W}_i(m)> K\right] \\
& \qquad \leq \frac{\veps}{2} + \pr\left[\sum_{i=1}^{\mathsf{R}(m) \wedge I} \mathsf{Y}_{\mathsf{b}^-_{i,m}}(m)>K
\right] + \pr\left[\exists i \le \mathsf{R}(m) \wedge I: \mathsf{Y}_{\mathsf{b}^-_{i,m}}(m)\geq d(\mathsf{a}_{i,m},\mathsf{b}_{i,m})\right].
\end{align*}
By Lemma~\ref{lem:gwsubtree}, the $(Y_{\mathsf{b}_{i,m}^{-}}(m),i\leq \mathsf{R}(m))$ are tight as $m \to \infty$, and thus so is the sum of at most $I$ of them:
\[
\pr\left[\sum_{i=1}^{\mathsf{R}(m) \wedge I} \mathsf{Y}_{\mathsf{b}^-_{i,m}}(m)>K\right]  \le \frac{\veps}{4}
\]
for all $m$, for $K$ large enough. For the final term, we may again adapt the argument from Lemma 19 of \cite{A-BBG12} to see that for each $i$, $m^{-1/2}d(\mathsf{a}_{i,m},\mathsf{b}_{i,m})$ converges in distribution, where $d$ denotes the graph distance in $\mathsf{T}_m.$ In particular there exists $\eta>0$ such that $\pr[d(\mathsf{a}_{i,m},\mathsf{b}_{i,m}) \leq m^{1/2}\eta]\leq \frac{\veps}{8I}.$ We then have
\[\pr \left[i\leq \mathsf{R}(m),\mathsf{Y}_{\mathsf{b}^-_{i,m}}(m)\geq d(\mathsf{a}_{i,m},\mathsf{b}_{i,m}) \right]\leq \frac{\veps}{8I} + \pr \left[i\leq \mathsf{R}(m),\mathsf{Y}_{\mathsf{b}^-_{i,m}}(m)\geq m^{1/2}\eta \right],\]
and by Lemma~\ref{lem:gwsubtree} again, $\pr[i\leq \mathsf{R}(m),\mathsf{Y}_{\mathsf{b}^-_{i,m}}(m)\geq m^{1/2}\eta] < \veps/8I$ for all $m$ sufficiently large, so that for such $m$, 
\[
\pr \left[\exists i\leq \mathsf{R}(m) \wedge I :\; \mathsf{Y}_{\mathsf{b}^-_{i,m}}(m) \geq d(\mathsf{a}_{i,m},\mathsf{b}_{i,m}) \right] \leq \frac{\veps}{4}.
\]
Combining all the terms yields
\[
\limsup_{m \to \infty} \pr\left[\sum_{i=1}^{\mathsf{R}(m)} W_i(m)>K\right]\leq \veps. \qedhere
\]
\end{proof}

\begin{proof}[Proof of Proposition~\ref{prop:surplus}] 
It remains to show that (\ref{eq:tightness}) holds. We use the change of measure to pass from $\mathsf{T}_m$ to $T_m$. Call $A(m,K)$ the event where $\sum_{i=1}^{R(m)} W_i(m)>K$ and $\mathsf{A}(m,K)$ the event where $\sum_{i=1}^{\mathsf{R}(m)} \mathsf{W}_i(m) >K$. Then we have
\[\pr[A(m,K)]= \frac{\E[(1-p)^{-a(\mathsf{T}_m)}\mathbf{1}_{\mathsf{A}(m,K)}]}{\E[(1-p)^{-a(\mathsf{T}_m)}]}
         \leq \frac{\sqrt{\E[(1-p)^{-2a(\mathsf{T}_m)}]}}{\E[(1-p)^{-a(\mathsf{T}_m)}]}\sqrt{\pr[\mathsf{A}(m,K)]}.\]
We know that $\E[(1-p)^{-2a(\mathsf{T}_m)}]$ is bounded and that $\E[(1-p)^{-a(\mathsf{T}_m)}]$ converges to a positive limit.  So by Lemma~\ref{lem:notilttightness}, we obtain
\[
\lim_{K \to \infty} \limsup_{m \to \infty} \pr[A(m,K)] = 0,
\]
as required.
\end{proof}

\subsection{Proof of Theorem~\ref{thm:main}}

We first prove that the convergence in Theorem~\ref{thm:main} occurs in the weaker product topology, namely that for any $k \in \N$,
\[
n^{-1/3} (C_1(n), C_2(n), \ldots, C_k(n)) \cv (\mathcal{C}_1, \mathcal{C}_2, \ldots, \mathcal{C}_k)
\]
with respect to $d_{\G}^k$. We will later improve this to a convergence with respect to $d$.

\subsubsection{Convergence in the product topology}

Let $(T_1^n,T_2^n,\ldots)$ be the forward exploration trees of $\vec{G}(n,p).$ We list them in decreasing order of their sizes $(Z_1^n,Z_2^n,\ldots)$, and recall that we write $(\| T_1^n\|, \|T_2^n\|,\ldots)$ for their heights. We also let $(Y_1^n,Y_2^n,\ldots)$ be the subgraphs of $\vec{G}(n,p)$ induced by the vertex-sets of these trees (which include both surplus and back edges).
By \cite{aldous1997}, we have the following convergence for the $\ell^2$ topology on sequences:
\begin{equation}\label{eq:cvaldous}
n^{-2/3}(Z_i^n,i \in \N) \cv  (\sigma_i,i\in\N),
\end{equation}
where $(\sigma_i,i\in\N)$ are the excursion lengths of $W_{\lambda}$ above its running infimum, sorted decreasingly.
Again, using Skorokhod's theorem, we may work on a probability space for which this convergence occurs almost surely. Moreover, conditionally on $(Z_1^n, Z_2^n, \ldots)$, the $(Y_i^n,i\in\N)$ are independent, each having the distribution of $X_{Z_n^i}$ as in Section~\ref{sec:onetree}.  Since $Z_n^i p^{2/3} \to \sigma_i$, we have that the rescaled strongly connected components of $Y_i^n$ converge in distribution to those of $\mathcal{M}_{\sigma_i}$, and this holds jointly for any finite set of indices $i$. Taking into account Proposition~\ref{prop:graphtocc}, the following proposition will give the convergence in Theorem~\ref{thm:main} for the product topology.

\begin{prop}\label{prop:aa} For all $k\in\N,$ we have 
\[\underset{K\to\infty}\lim \, \pr[\forall i\leq k,\exists j\leq K:\;\mathcal{C}_i\subseteq \D_j]=1\]
and,
\[\underset{K\to\infty}\lim\, \underset{n\to\infty}\liminf \, \pr[\forall i\leq k,\exists j\leq K:\; C_i(n) \subseteq Y_j^n]=1.\]
\end{prop}

Informally, Proposition~\ref{prop:aa} states that, with high probability, large strongly connected components of $\vec{G}(n,p)$ and $\mathcal{D}$ will only be found in large trees of the forward depth-first forest, making the ordering of both trees and strongly connected components by their lengths compatible. Its proof relies on two lemmas.

\begin{lem}\label{lem:smallcontinuous}
As $\sigma\to 0,$ we have
\begin{equation}\label{eq:smallcontinuous1}
\pr[\mathcal{M}_{\sigma}\text{ has a complex component}]= O(\sigma^3).\end{equation}
For all $\veps>0,$ we have as $\sigma\to 0$
\begin{equation}\label{eq:smallcontinuous2}
\pr[\|\T_{\sigma}\|\ge \veps] =O (\sigma^3).\end{equation}
Consequently, for all $\veps>0,$
\begin{equation}\label{eq:smallcontinuous3}
\pr\Big[\mathcal{M}_{\sigma}\text{ has a component with length greater than }\veps \Big]=  O(\sigma^3).\end{equation}
\end{lem}

\begin{lem}\label{lem:smalldiscrete}
There exists $C>0$ such that, for all $n$ large enough and $1\le m\le n^{2/3},$
\begin{equation}\label{eq:smalldiscrete1}\pr\Big[X_m\text{ has a complex component}\Big]\leq C\frac{m^3}{n^{2}}\end{equation}
and
\begin{equation}\label{eq:smalldiscrete15}
\pr\Big[X_m\text{ has a component which contains a surplus edge}\Big]\leq C\frac{m^3}{n^{2}}.\end{equation}
Moreover, for all $\veps>0,$ there exists $C>0$ such that, for $n$ large enough, and $1\leq m\leq n^{2/3},$
\begin{equation}\label{eq:smalldiscrete2}
\pr[\|T_m\| \geq n^{1/3}\veps]\leq C\frac{m^2}{n^{4/3}}.\end{equation}
Consequently, for all $\veps>0,$ there exists $C>0$ such that, for $n$ large enough,
\begin{equation}\label{eq:smalldiscrete3}\pr\Big[X_{m}\text{ has a component with length greater than }n^{1/3}\veps \Big]\leq C\frac{m^2}{n^{4/3}}.\end{equation}
\end{lem} 
Note that for both of these lemmas, the final statement is a consequence of the previous ones by noticing that any component consisting of a single ancestral cycle has length smaller than the height of the tree (plus one in the discrete case).

\begin{proof}[Proof of Lemma~\ref{lem:smallcontinuous}]
By Proposition~\ref{prop:Poissonbounds},
\[
\pr\Big[\mathcal{M}_{\sigma}\text { has a complex component}\Big] \leq \pr\big[N_{\sigma}^a\geq 2 \text{ or } N_{\sigma}^b\geq 1] =O(\sigma^3).
\]
Recalling that the height $\|\T_{\sigma}\|$ has the same distribution as $\sup 2\tilde{\exc}^{(\sigma)}$ and that it has exponential moments \cite{kennedy76}, we have, for $\sigma<1,$
\begin{align*}
\pr[\|\T_{\sigma}\| > \veps]&= \frac{\E\left[\mathbf{1}_{\{\sup \mathbf{e}\geq \veps/2\sqrt{\sigma}\}} \exp \left(\sigma^{3/2}\int_0^1\exc(x)dx \right) \right]}{\E\big[\exp \left( \sigma^{3/2}\int_0^1\exc(x)dx \right)\big]}\\
        &\le \E[e^{\sup \exc-\veps/2\sqrt{\sigma}}\,e^{\sigma^{3/2}\sup\exc}]\leq \E[e^{2\sup \mathbf{e}}]e^{-\veps/2\sqrt{\sigma}}=O(\sigma^3).
\end{align*}
Note we have used the bound $\mathbf{1}_{\{\sup \mathbf{e}\geq \veps/2\sqrt{\sigma}\}}\leq e^{\sup \exc-\veps/2\sqrt{\sigma}}$. This proves~(\ref{eq:smallcontinuous1}) and~(\ref{eq:smallcontinuous2}); (\ref{eq:smallcontinuous3}) then follows.
\end{proof}

For Lemma~\ref{lem:smalldiscrete}, we require some preliminary bounds on the height and area of $T_m$.

\begin{lem}\label{lem:boundsheightarea} There exists a constant $M>0$ such that, for all $n$ large enough such that $1/(2n)<p<2/n$ and all 
$1\leq m \leq n^{2/3},$
\begin{align}
\label{eq:heightbound} \E[\|T_m\|^4] & \leq M  m^2 \\
\intertext{and}
\label{eq:areabound} \E[(a(T_m))^2] & \leq M m^3.
\end{align}
\end{lem}
\begin{proof} Lemma 25 from \cite{A-BBG12} gives $\E[\|T_m\|^4]\leq M \cdot \max (m^6n^{-4},1) \cdot m^2$ for all $n$ large enough and $m\leq n,$ and restricting ourselves to $m\leq n^{2/3}$ yields (\ref{eq:heightbound}).

For (\ref{eq:areabound}), we follow the beginning of the proof of Lemma 25 from \cite{A-BBG12}. Let $q=\max(m^{-3/2},p)$.  Then (\ref{eq:defbias}) and Markov's inequality together yield
\begin{align*}
\pr[a(T_m)> xm^{3/2}]   &\leq \frac{\E[(1-q)^{-a(T_m)}]}{(1-q)^{-xm^{3/2}}} \leq \frac{\E[((1-p)(1-q))^{-a(\mathsf{T}_m)}]}{(1-q)^{-xm^{3/2}}} \leq \frac{\E[(1-q)^{-2a(\mathsf{T}_m)}]}{(1-q)^{-xm^{3/2}}}.
\end{align*}
From Lemma 14 in \cite{A-BBG12}, we obtain that $\E[(1-q)^{-2a(\mathsf{T}_m)}]\leq K \exp{4\kappa\delta^2}$ where $\delta=\max(2m^{3/2}/n,1).$ Since $qm^{3/2}\geq \delta/4$, we get
\[\pr[a(T_m)> xm^{3/2}]\leq Ke^{4\kappa\delta^2-x\delta/4},\]
and for $1\leq m\leq n^{2/3},$ we have $1\leq \delta \leq 2,$ so that
\[\pr[a(T_m)> xm^{3/2}]\leq Ke^{64\kappa-x/4}.\]
It follows that
\[
\E\left[\frac{(a(T_m))^2}{m^3} \right]\leq Ke^{64\kappa} \int_0^{\infty} e^{-\sqrt{x}/4}\mathrm d x = 32Ke^{64\kappa},
\]
which completes the proof.
\end{proof}
\begin{proof}[Proof of Lemma~\ref{lem:smalldiscrete}]
We take $n$ large enough for (\ref{eq:heightbound}) and (\ref{eq:areabound}) to hold, and $m\leq n^{2/3}.$
Notice first that~(\ref{eq:smalldiscrete2}) follows from (\ref{eq:heightbound}) and Markov's inequality:
\[\pr[\|T_m\|\geq n^{1/3}\veps]\leq \frac{\E[\|T_m\|^4]}{\veps^4n^{4/3}}\leq \frac{M m^2}{\veps^4n^{4/3}}.\]
We now want to show that the probability that $X_m$ contains a strongly connected component which is complex or features surplus edges is also bounded by $m^3n^{-2}$. Such a component can only arise if one of the following four events occurs:
\begin{align*}
A_m & = \{X_m \text{ has at least two ancestral back edges.}\} \\
B_m & =\{X_m \text{ has one ancestral back edge, and at least one other back edge which points} \\
& \qquad \text{ inside the created cycle.}\} \\
C_m & =\{X_m \text{ has at least two surplus edges.}\} \\
D_m & =\{X_m \text{ has one surplus edge $(a,b)$ and at least one back edge pointing to} \\
& \qquad \text{ an ancestor of }a.\}
\end{align*}
We will bound the probabilities of each of these events separately. 

Conditionally on the tree $T_m$, the number of ancestral back edges in $X_m$ has distribution $\mathrm{Bin}(S_m,p),$ where is $S_m$ the sum of the heights of all vertices in $T_m.$ By using the well-known stochastic domination of $\mathrm{Bin}(k,p)$ by $\mathrm{Poi}(-k\log(1-p))$ and the fact that $\pr[\mathrm{Poi}(\mu)\geq 2]\leq \mu^2$, we have
\[\pr[A_m\mid T_m]\leq (-S_m \log(1-p))^2 \leq M(S_m p)^2. \]
From now on, the constant $M$ can vary from line to line, but never depends on $n$ or $m.$

Since $S_m\leq m \|T_m\|,$ by using (\ref{eq:heightbound}) again, for $n$ large enough we end up with
\[
\pr[A_m ] \leq M\frac{m^2}{n^2}\E[\|T_m\|^2] \leq M\frac{m^3}{n^2}.
\]
Given that there is exactly one ancestral back edge in $X_m$, the number of back edges which point back into the cycle created is stochastically dominated by $\mathrm{Bin}(m\|T_m\|,p).$ Hence we have
\begin{align*}
\pr[B_m \mid T_m] &\leq p S_m (1-p)^{S_m-1}(1-(1-p)^{m\|T_m\|}) \\
                               &\leq M p S_m (m \|T_m\| \log(1-p)) \\
                               &=M n^{-2}(m \|T_m\|)^2.
\end{align*}
This is the same bound as above, thus leading to
\[
\pr[B_m]\leq M\frac{m^3}{n^2}
\]
Since the number of surplus edges has distribution $\mathrm{Bin}(a(T_m),p),$ we get $\pr[C_m \mid T_m] \leq M p^2 a(T_m)^2$ and 
\[\pr[C_m]\leq Mn^{-2}\E[a(T_m)^2].\] A similar argument as for $B_m$ also yields 
\[\pr[D_m]\leq M n^{-2}m \E[\|T_m\| a(T_m)]\leq Mn^{-2}m \sqrt{\E[\|T_m\|^2]}\sqrt{\E[a(T_m)^2]},\]
and an application of (\ref{eq:areabound}) concludes the proof.
\end{proof}

We can now prove the proposition.

\begin{proof}[Proof of Proposition~\ref{prop:aa}]
Fix $k\in\N$ and $\eta>0$, and let $\veps>0$ be small enough that 
\[
\pr[\mathrm{len}(\mathcal{C}_k) >\veps] > 1-\eta.
\]
By Lemmas~\ref{lem:smallcontinuous} and~\ref{lem:smalldiscrete}, there exists $C>0$ such that, for all $K\in\N,$
\[\pr\Big[\exists i> K: \mathcal{D}_{i}\text{ contains a component with length greater than }\veps \Big]\leq \pr[\sigma_{K+1}>1]+C\E\left[\sum_{i> K}\sigma_i^3\right]\]
and
\begin{align*}
& \pr\Big[\exists i> K: Y_i^n\text{ contains a component with length greater than }n^{1/3}\veps \Big] \\
& \qquad \leq \pr[Z_{K+1}^n > n^{2/3}] +C\E\left[\sum_{i> K}\frac{(Z_i^n)^2}{n^{4/3}}\right].
\end{align*}
 By Proposition~\ref{prop:excursionlengths} and (\ref{eq:cvaldous}), there exists $K$ sufficiently large that both of these are smaller than $\eta.$ Then
\[\pr[\mathcal{C}_1, \ldots, \mathcal{C}_k \text{ are in  }\mathcal{D}_1,\ldots,\mathcal{D}_K,  \ \mathrm{len}(\mathcal{C}_k) > \veps]\geq 1-2\eta.\]
From the fact that $(n^{-1/3}Y_1^n,\ldots,n^{-1/3}Y_K^n) \cv (\mathcal{D}_1,\ldots,\mathcal{D}_N)$, we deduce that, for $n$ greater than some $n_0\in\N$,
\[\pr[C_1(n), \ldots, C_k(n) \text{ are in }Y_1^n,\ldots,Y_K^n, \ \mathrm{len}(C_k(n)) > \veps n^{1/3}]\geq 1-3\eta\]
and hence
\[\pr[C_1(n),\ldots,C_k(n) \text{ are in } Y_1^n,\ldots, Y_K^n]\geq 1-4\eta. \qedhere
\]
\end{proof}

\subsubsection{Controlling the tail}
The proof of Theorem~\ref{thm:main} will be completed if we can show that, for all $\veps>0,$
\[\underset{k\to\infty}\lim \pr\left[\sum_{i=k+1}^{\infty} d_{\G}(\mathcal{C}_i,\mathfrak{L}) >\veps\right]=0\]
and 
\[\underset{k\to\infty}\lim\, \underset{n\to\infty}\limsup\, \pr\left[\sum_{i=k+1}^{\infty} d_{\G}(C_i(n),\mathfrak{L}) >n^{1/3}\veps\right]=0.\]

Fix $\eta>0.$ For $k\in \N,$ let $Q(k)$ be the largest integer such that 
\[
\pr \left[\text{all components in $\mathcal{D}_1,\ldots,\mathcal{D}_{Q(k)}$ have lengths exceeding $\mathrm{len}(\mathcal{C}_k)$} \right] > 1-\eta.
\]
Then by Proposition~\ref{prop:aa} and the convergence of $n^{-1/3} C_k(n)$ to $\mathcal{C}_k$, it also holds that, for $n$ large enough, all the components of $Y_1^n,\ldots,Y^n_{Q(k)}$ have lengths exceeding that of $C_k(n)$ with probability at least $1-2\eta.$ Thus we have
\[\pr\left[\sum_{i=k+1}^{\infty} d_{\G}(\mathcal{C}_i,\mathfrak{L}) >\veps\right]\leq \eta + \pr\left[\left(\sum_{i=Q(k)+1}^{\infty}\sum_{j: \, \mathcal{C}_j\subset \mathcal{D}_i}d_{\G}(\mathcal{C}_j,\mathfrak{L})\right) >\veps\right]\]
and similarly
\begin{align*}
& \pr\left[\sum_{i=k+1}^{\infty} d_{\G}(C_i(n),\mathfrak{L}) >n^{1/3}\veps\right] \\
& \qquad \leq 2\eta + \pr\left[\left(\sum_{i=Q(k)+1}^{\infty}\sum_{j: \, C_j(n)\subset X_i^n}d_{\G}(C_j(n),\mathfrak{L})\right) >n^{1/3}\veps\right].
\end{align*}
Note that $Q(k) \to \infty$ as $k \to \infty$: indeed, it is non-decreasing, and so if did possess a finite limit $Q$, then the probability of $\mathcal{D}_1,\ldots,\mathcal{D}_{Q+1}$ containing a smallest component of $\mathcal{C}$ would be at least $\eta,$ a contradiction since there is no smallest component. It is therefore enough to prove that
\[\underset{N\to\infty}\lim \pr\left[\sum_{i=N+1}^{\infty}\sum_{j: \, \mathcal{C}_j\subset \mathcal{D}_i} d_{\G}(\mathcal{C}_j,\mathfrak{L}) >\veps\right]=0\]
and
\[\underset{N\to\infty}\lim\, \underset{n\to\infty}\limsup\, \pr\left[\sum_{i=N+1}^{\infty} \sum_{j: \, C_j(n)\subset X_i^n} d_{\G}(C_j(n),\mathfrak{L}) >n^{1/3}\veps\right]=0.\]
However, by~(\ref{eq:smallcontinuous1}),~(\ref{eq:smalldiscrete1}) and~ (\ref{eq:smalldiscrete15}), for $N$ large enough, all the components contained in $(\mathcal{D}_i,i\ge N+1)$ are single ancestral cycles with probability at least $1-\eta,$ and for $n$ large enough, this also holds for those contained in $(Y_i^n,i\ge N+1)$. Noting that such components have length at most the height of the underlying tree (plus one in the discrete case), and that their number is at most the number of ancestral back edges, we are reduced to proving the following statements:
\begin{equation}\label{eq:smalldiamcont}
\underset{K\to\infty}\lim \pr\left[\sum_{i=K+1}^{\infty} N_{\sigma_i}^a\|\T_{\sigma_i}\| >\veps\right]=0
\end{equation}
and 
\begin{equation}\label{eq:smalldiamdisc}
\underset{K\to\infty}\lim\, \underset{n\to\infty}\limsup\, \pr\left[\sum_{i=K+1}^{\infty}A_i^n \|T_i^n\| >n^{1/3}\veps\right]=0,
\end{equation}
where $A_i^n$ is the number of ancestral back edges in $Y_i^n.$  These may be obtained using the following lemma.

\begin{lem}
\begin{itemize}
\item[$(i)$] There exists $C>0$ such that, for $\sigma<1,$
\[\E[N_{\sigma}^a \|\T_{\sigma}\|\,]\leq C \sigma^2.\]
\item[$(ii)$] There exists $C>0$ such that, for $n$ large enough, and $1\leq m\leq n^{2/3},$
\[\E\left[A_m \|T_m\|\,\right]\leq C \frac{m^2}{n},\]
where $A_m$ is the number of ancestral back edges in $X_m.$
\end{itemize}
\end{lem}

\begin{proof} Part $(i)$ is straightforward: assuming $\|\T_{\sigma}\|$ and $N_{\sigma}^a$ are built from a tilted excursion $\tilde{\exc}^{(\sigma)},$ and remembering that $\E[N_{\sigma}^a\mid \tilde{\exc}^{(\sigma)}] = \int_0^\sigma 2\tilde{\exc}^{(\sigma)}(t)\mathrm dt$, we have
\begin{align*}
\E[N_{\sigma}^a \|\T_{\sigma}\|\,]&=\E \left[\sup 2\tilde{\exc}^{(\sigma)}\int_0^{\sigma}2\tilde{\exc}^{(\sigma)}(t)\mathrm dt \right] \\
&\leq 4\sigma\E[\sup (\tilde{\exc}^{(\sigma)})^2] \\
 &\leq 4\sigma\frac{\E\left[\sup (\sqrt{\sigma}\exc)^2  \exp \left(\sigma^{3/2}\int_0^1\exc(t)dt  \right) \right]}{\E\big[\exp \left( \sigma^{3/2}\int_0^1\exc(t)dt \right)\big]} \\
 &\leq 4\sigma^2 \E \left[e^{\int_0^{1}\exc(t)\mathrm dt}\sup (\exc)^2 \right],
\end{align*}
the latter expectation being finite (see the proof of Lemma~\ref{lem:smallcontinuous}).
For part $(ii),$ recall that, conditionally on $T_m$ the distribution of $A_m$ is stochastically dominated by $\mathrm{Bin}(m \|T_m\|,p).$ Thus, we have
\[\E[A_m \|T_m\|]\leq p\E[m \|T_m\|^2],\]
and applying Lemma~\ref{lem:boundsheightarea} concludes the proof.
\end{proof}

We leave the straightforward adaptation of the arguments used for Proposition~\ref{prop:aa} to prove (\ref{eq:smalldiamcont}) and (\ref{eq:smalldiamdisc}) to the reader.  This completes the proof of Theorem~\ref{thm:main}.

\section{Further properties of the scaling limit} \label{sec:furtherprops}
We write $\mathcal{C}$ for the list of strongly connected components of $\mathcal{D},$ and $\mathcal{C}_{\sigma}$ for that of $\mathcal{M}_{\sigma},$ in decreasing order of length. Let also $\mathcal{C}_{\mathrm{cplx}}$ be the list of \emph{complex} components of $\mathcal{C},$ i.e.\ those that are not cycles, also in decreasing order of length. We have not yet been able to find the exact distribution of $\mathcal{C}$ and $\mathcal{C}_{\sigma}$ for $\sigma>0$: this will be the subject of future research. However, we show here that $\mathcal{C}_{\sigma}$ and  $\mathcal{C}_{\mathrm{cplx}}$ both have a positive probability of being equal to any appropriate fixed family of directed multigraphs.

For sequences $(G_1,\ldots,G_k)$ and $(H_1,\ldots,H_j)$ of directed multigraphs, we write $(G_1,\ldots,G_k)\equiv (H_1,\ldots,H_j)$ if $j=k$ and $G_i$ is isomorphic to $H_i$ for each $i\leq j.$ We extend this notation naturally to the case where  one or both of the sequences has edge lengths by simply ignoring the edge lengths.

\begin{prop}\label{prop:fullsupport} Let $G_1,\ldots,G_k$ be a finite sequence consisting of $3$-regular strongly connected directed multigraphs or loops. We have
\[\pr[\mathcal{C}_{\sigma}\equiv(G_1,\ldots,G_k)]>0.\]
Assuming that $G_1,\ldots,G_k$ are all complex, we also have
\[\pr[\mathcal{C}_{\mathrm{cplx}}\equiv (G_1,\ldots,G_k)]>0.\]
Let $(e_i,1\leq i\leq K)$ be an arbitrary ordering of the edges of $(G_1,\ldots,G_k).$ Then, conditionally on $\mathcal{C}_{\sigma}\equiv(G_1,\ldots,G_k)$ (resp. $\mathcal{C}_{\mathrm{cplx}}\equiv (G_1,\ldots,G_k)$), $\mathcal{C}_{\sigma}$ (resp. $\mathcal{C}_{\mathrm{cplx}})$ gives lengths $(\ell(e_i),1\leq i\leq K)$ to these edges, and their joint distribution has full support in 
\[\left\{\mathbf{x}=(x_1,\ldots,x_K)\in\R_+^K: \forall\, 1\leq i\leq k-1, \sum_{j:e_j\in E(G_i)}x_j \geq \sum_{j:e_j\in E(G_{i+1})}x_j\right\}.\]
\end{prop}

\noindent\textbf{Constructing $3$-regular directed multigraphs from trees and back edges.} First, we want to show that any of the graphs in which we are interested can be constructed by a procedure which adds back edges to a plane tree.  We set this up in a discrete framework. Let $\mathsf{t}$ be a discrete plane tree whose vertices have outdegrees in $\{0,1,2\}.$  We think of this as a directed graph, with edges pointing away from the root. We assume that $\mathsf{t}$ has as many leaves as internal vertices of outdegree one, which we call $x_1,\ldots,x_n$ and $y_1,\ldots,y_n$ respectively, in the planar order.  We assume, moreover, that for each $i \ge 1$, the internal vertex $y_i$ is visited before the leaf $x_i$ in the depth-first exploration.  By identifying $x_i$ and $y_i$ for all $i,$ we obtain a directed graph, whose strongly connected components we then extract.  Each strongly connected component will have exactly one vertex of degree $2$, which we erase, merging its two incident edges. The result is a set of $3$-regular strongly connected directed multigraphs. The next lemma asserts that any appropriate collection of such multigraphs can be obtained by this procedure, and Figure~\ref{fig6} provides an example.

\begin{figure}[h]
\includegraphics[scale=1]{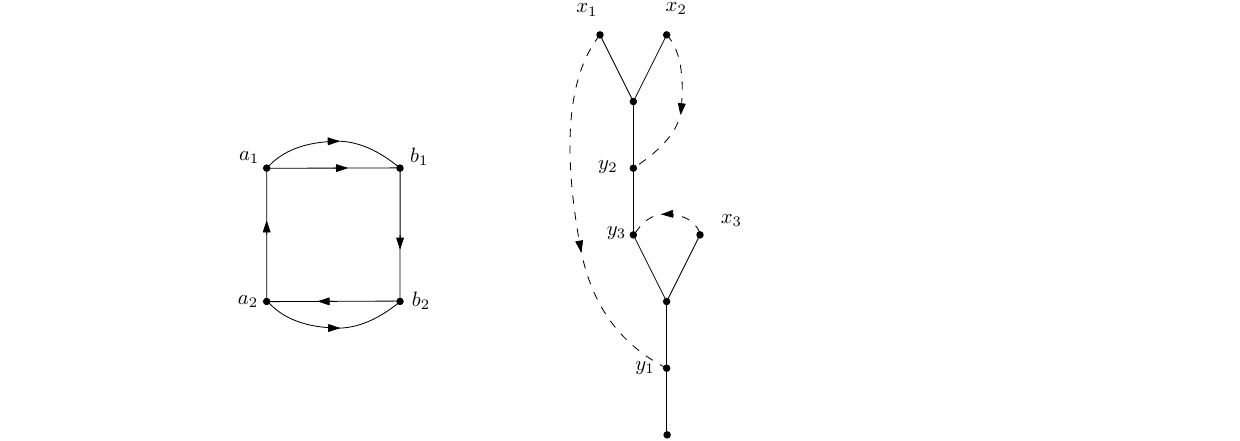}
\caption{Obtaining a $3$-regular connected directed multigraph from a tree with backward identifications. The tree was built using the method presented in the proof of Lemma~\ref{lem:existstree}.}
\label{fig6}
\end{figure}

\begin{lem}\label{lem:existstree} For any $(G_1,\ldots,G_k),$ there exist a discrete plane tree $\mathsf{t}$ and pairings $(x_i,y_i)$ such that the above construction results in $(G_1,\ldots,G_k).$
\end{lem}

\begin{proof}
Notice first that we can focus on the case where $k=1$. Once this case is treated, the general case can be solved by taking a tree $\mathsf{t}$ which contains distinct subtrees corresponding to each $G_i.$

So let $G$ be a fixed strongly connected 3-regular directed multigraph. Noticing that it cannot have vertices with outdegree $0$ or $3$, and that the sum of the outdegrees of all the vertices is equal to that of all the indegrees, we deduce that there exists $n\in\N$ such that $G$ has $n$ vertices with indegree $1$ and outdegree $2$, and $n$ vertices with indegree $2$ and outdegree $1$.  Let $a_1,\ldots,a_n$ be the former and $b_1,\ldots,b_n$ the latter, for any ordering such that the edge $(b_1,a_1)$ exists.

We will give a method to construct the necessary plane tree as well as the backward links between leaves and edges. At each step, $\mathsf{t}$ will contain a certain number of vertices of $G,$ as well as some ``open" edges, which have their tails at points in $\mathsf{t}$ but are missing their heads.

Start with $\mathsf{t}$ initially containing three vertices: a root with outdegree 1, its child (which we arbitrarily call $\rho_0$) which has outdegree $1$ as well, and its next neighbour $a_1,$ from which originate two open edges. At each step of the algorithm, let $z$ be the leftmost of the deepest vertices of $\mathsf{t}$ which have open edges, choose any edge of $G$ starting at $z$ which is not yet featured in $\mathsf{t}$, call $u$ the head of that edge, and do the following:
\begin{itemize}
\item If $u$ is not already in $\mathsf{t},$ add it at the end of the leftmost open edge, and add one or two open edges at $u$ corresponding to its outdegree in $G$. The edge $(z,u)$ is then a tree edge in $\mathsf{t}.$
\item If $u$ is already in $\mathsf{t}$ but $u\neq a_1,$ add a leaf at the end of the leftmost open edge, label that leaf $x_j$ for the smallest available $j$ and let also $u=y_j.$ The edge $(z,u)$ is then featured in $\mathsf{t}$ as the tree edge $(z,x_j),$ identifying $x_j$ with $u.$
\item If $u=a_1,$ put a leaf at the end of the leftmost open edge, label that leaf $x_j$ for the smallest available $j$, and let $y_j=\rho_0.$ The edge $(z,u)$ is then featured in $\mathsf{t}$ as the merging of the tree edges $(z,x_j)$ and $(\rho_0,a_1),$ identifying $x_j$ with $\rho_0.$
\end{itemize}
Note that this algorithm terminates, and that identifying the pairs $(x_i,y_i)$ in $\mathsf{t}$ and removing the root (which is not in its strongly connected component) and $\rho_0$ (which has degree $2$ in the strongly connected component) gives us $G.$

Moreover, by construction, the successive vertices appearing as $z$ follow the planar ordering of $\mathsf{t}$. This means that at any step, any other vertex of $\mathsf{t}$ can be found earlier than $z$ in the contour process, and thus in every pair $(x_i,y_i),$ the vertex $y_i$ is seen earlier than $x_i$ in the exploration process, and the identifications indeed go backwards. This completes the proof.
\end{proof}

\noindent\textbf{The marked tree has full support.} If $T$ is a discrete plane tree and $\T$ is a discrete plane tree with edge lengths (equivalently an $\R$-tree with finitely many leaves which are ordered), we write $\T\equiv T$ if the discrete plane structure underlying $\T$ is $T$. If $\T\equiv T$ then the lengths of the edges of $\T$, in planar order, form a vector in $\R_+^k$ where $k$ is the number of edges of $T$.

Let $T$ be a fixed binary rooted discrete plane tree with $n\in\N$ leaves. For an excursion function $f: [0,\sigma]\to \R_+,$ we let $D_T(f)$ be the set of increasing sequences $\mathbf{t}=(t_1,\ldots,t_n)\in[0,\sigma]^n$ such that the $\T_f(t_1,\ldots,t_n)\equiv T.$ This is an open subset of $[0,\sigma]^n$ which can be written explicitly as 
\begin{align*}
D_T(f)= & \left\{\mathbf{t}\in[0,\sigma]^n: \; t_1<t_2\ldots <t_n\text{ and } \forall k\in\{3,\ldots,n\},\, \right. \\
& \left. \qquad \qquad \qquad \hat{f}(t_{i(k)},t_{k-1})< \hat{f}(t_{k-1},t_k) < \hat{f}(t_{j(k)},t_{k-1}) \right\}. 
\end{align*}
Here the indices $i(k)$ and $j(k)$ are defined as follows, and illustrated by Figure~\ref{fig7}. Let $L_1,\ldots,L_n$ be the leaves of $T$ in planar order (we add $L_{0}=\rho$ for the sake of convenience). For $k\in \{3,\ldots,n\}$, we then take $i(k)<j(k)$ to be any two integers in $\{0,1,2,\ldots,k-1\}$ such that, on the path $\llbracket \rho, L_{k-1}\rrbracket,$ the two points $L_{i(k)} \wedge L_{k-1}$ and $L_{j(k)} \wedge L_{k-1}$ are respectively maximal and minimal such that $L_{i(k)} \wedge L_{k-1}\leq L_{k-1} \wedge L_k \leq L_{j(k)}\wedge L_{k-1}$ for the genealogical/planar order. 

\begin{figure}[h]
\centering
\includegraphics[scale=1]{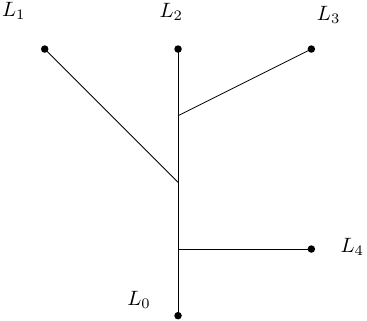}
\caption{For this tree, $i(3)=1$, $j(3)=2$, $i(4)=0$, and $j(4)=1.$ Given an excursion function $f,$ a sequence $t_1<t_2<t_3<t_4$ will then be in $D_T(f)$ iff $\hat{f}(t_1,t_2)<\hat{f}(t_2,t_3)<f(t_2)$ and $0<\hat{f}(t_3,t_4)<\hat{f}(t_1,t_4).$}
\label{fig7}
\end{figure}

\begin{lem}\label{lem:chargetree} We have 
\begin{align*}
\pr[\T_f^{\mathrm{mk}}=T]= \int_{\mathbf{t}\in D_T(f)}&\mathrm{d}\mathbf{t} \prod_{k=1}^n \left(\sum_{i=1}^k f(t_i) - \hat{f} (t_{i-1},t_i)\right) \\ 
    &\exp{\left(-\int_0^{\sigma} \Bigg(f(t)-\hat{f}(t_{I(t)},t) + \sum_{i=1}^{I(t)}f(t_i)-\hat{f}(t_{i-1},t_i)\Bigg)\mathrm{d}t\right)},
\end{align*}
where $t_0=0$ and, for $t\in[0,\sigma]$, $I(t)=\max\{i: t_i<t\}.$

Moreover, if we take $f=2\tilde{\exc}^{(\sigma)}$ for $\sigma>0$, then 
\[\pr[\T_{\sigma}^{\mathrm{mk}}\equiv T]>0, \]
and conditionally on $\T_{\sigma}^{\mathrm{mk}}\equiv T$, the joint distribution of the edge lengths of $\T_{\sigma}^{\mathrm{mk}}$ has full support in $\R_+^{2n-1}.$
\end{lem}

\begin{proof}
The first statement comes from Lemma~\ref{lem:explicit}. For the second statement, we use a comparison with the scaling limit of the undirected random graph. Specifically, Lemma 10 of \cite{A-BBG10} gives the joint distribution of the tree shape and the edge lengths in the subtree of $\mathcal{T}_{\sigma}$ spanned by the root and a random collection of leaves obtained as the projection of a Poisson point process on $[0,\sigma]$ with intensity $\tilde{\exc}^{(\sigma)}(\cdot)$ onto the tree.\footnote{The sampled leaves in the undirected graph setting come from a Poisson point process with intensity $\tilde{\exc}^{(\sigma)}(\cdot)$ rather than the intensity $2\tilde{\exc}^{(\sigma)}(\cdot)$ we have in our construction for the directed graph.  This is because (as seen in~\cite{A-BBG12}) in the setting of the undirected graph the identifications arise as the limit of the surplus edges: the number of potential surplus edges originating at a single vertex is given not by the height of the vertex but rather by the number of vertices sitting on the stack in the depth-first exploration (the so-called \emph{depth-first walk}).  The depth-first walk is asymptotically half the size of the height, and so has scaling limit $\tilde{\exc}^{(\sigma)}$ rather than $2\tilde{\exc}^{(\sigma)}$.} In particular, the probability that this procedure gives the tree shape $T$ and that the lengths of the edges (in planar order) lie in an open set $A\subset \R_+^{2n-1}$ is positive, that is
\[\E\left[\int_{\mathbf{t}\in D_T(2\tilde{\exc}^{(\sigma)})}\mathrm{d}\mathbf{t} \mathbf{1}_{\{(2\tilde{\exc}^{(\sigma)}(\mathbf{t}))\in A'\}}\prod_{k=1}^n  \tilde{\exc}^{(\sigma)}(t_k) \exp{\left(-\int_0^{\sigma} \tilde{\exc}^{(\sigma)}(t)\mathrm{d}t\right)}\right]>0,\]
where $A'\in \R_+^n$ is the open set such that the heights the leaves of $T$ are in $A'$ iff its edge lengths are in $A$.
This implies that $\E[G]>0$ where 
\[
G=\int_{\mathbf{t}\in D_T(2\tilde{\exc}^{(\sigma)})}\mathrm d \mathbf t\mathbf{1}_{\{(2\tilde{\exc}^{(\sigma)}(\mathbf{t}))\in A'\}}\prod_{k=1}^n \left(\sum_{i=1}^k 2 \tilde{\exc}^{(\sigma)}(t_i) \right),
\]
(since $G$ is larger than the random variable in the expectation above).  We then have
\begin{align*}
\pr[\T_{\sigma}^{\mathrm{mk}}\equiv T,\text{lengths in }A] &\geq \E\left[G\exp{(-\sigma(n+1)\sup \tilde{\exc}^{(\sigma)})}\right],
\end{align*}
and this is positive since $\sup \tilde{\exc}^{(\sigma)}$ is a.s.\ finite.
\end{proof}

\begin{proof}[Proof of Proposition~\ref{prop:fullsupport}] 
We first show the result for $\mathcal{C}_{\sigma}$. Let $\mathsf{t}$ and $((x_i,y_i),i\in\{1,\ldots,n\})$ be the discrete tree and pairing of leaves and outdegree-$1$ vertices given by Lemma~\ref{lem:existstree}. Moreover, let $T$ be obtained from $\mathsf{t}$ by erasing the vertices of degree $2$, and merging their adjacent edges. Let
\begin{itemize}
\item $(e_1,\ldots,e_K)$ be the edges of $(G_1,\ldots,G_k),$ in any order;
\item $(e_1,\ldots,e_K,\e_{K+1},\ldots,e_N)$ be those of $\mathsf{t},$ in any order completing the previous one;
\item $(f_1,\ldots,f_M)$ be those of $T,$ in planar order.\footnote{Note that in fact we have $M=2n-1$, $N=3n-1$ and $K=3(n-k)+k',$ where $k'$ is the number of unicycles amongst $(G_1,\ldots,G_k)$; however, this fact is not useful here.}
\end{itemize}

By construction, each edge of $(G_1,\ldots,G_k),$ is an edge of $\mathsf{t}$, justifying the notation for the edges of $\mathsf{t}.$ Moreover, each edge of $T$ is obtained by merging edges of $\mathsf{t}$, so there exists a partition of $\{1,\ldots,N\}$ with blocks $(S(i), 1 \le i \le M)$ such that for each $i\in\{1,\ldots,M\}$, $f_i$ is obtained by merging $e_j$ for $j\in S(i).$ For $i\in\{1,\ldots,n\}$, let $e_T(y_i)$ be the edge of $T$ containing $y_i$.
Given this information, we call a collection of positive lengths $\ell(e_i)$, and $\ell(f_i)$ 
such that $\ell(f_i)=\sum_{j\in S(i)} \ell(e_j)$ an \emph{admissible length assignment}.

Recall that, from the construction given in Section~\ref{sec:continuousbackedges}, conditionally on $\T_{\sigma}^{\mathrm{mk}}$ with leaves $L_1,\ldots,L_p,$ the marked internal points $z_1,\ldots,z_p$ are independent and, for each $j$, $z_j$ is uniform on $\cup_{k=1}^j \llbracket \rho,L_k\rrbracket.$ If $\T_{\sigma}^{\mathrm{mk}}\equiv T$ then this gives rise to a length assignment $\ell$ on $T,$ and we have
\[\pr\left[z_j \in e_T(y_j),\, \forall j\in\{1,\ldots,n\}\mid \T_{\sigma}^{\mathrm{mk}}\equiv T\right]\geq 
\prod_{j=1}^n\frac{\ell(g(y_j))}{\mathrm{len}(\T_{\sigma}^{\mathrm{mk}})}.\]
Moreover, conditionally on the event $\{z_j \in e_T(y_j),\, \forall j\in\{1,\ldots,n\}, \T_{\sigma}^{\mathrm{mk}}\equiv T\}$, for any edge $f_i$ of $T$, the probability that  $z_j$, for $j$ such that $y_j\in f_i$, are in the right order on $f_i$ is $\frac{1}{|S(i)|!}.$ If this occurs, then it gives rise to a length assignment $\ell$ on $\mathsf{t}$ as well, making the whole length assignment admissible. We then have  $(\ell(e_j),j\in S(i))=(D_1(i)\ell(f_i),\ldots,D_{|S(i)|}(i)\ell(f_i))$ where $D(i)=(D_1(i),\ldots,D_{|S(i)|}(i))\in \Delta_{|S(i)|}$ has the Dirichlet$(1,\ldots,1)$ distribution on the $(|S(i)|-1)$-dimensional simplex $\Delta_{|S(i)|}$. These events occur independently for different $i\in\{1,\ldots,M\}.$

Let $A$ be an open set in $\R_+^K.$ Take open sets $B\subset \R_+^M$ and $C_i\in \Delta_{|S(i)|}$ for $i\in \{1,\ldots,M\}$ such that, for any admissible length assignment, if $(\ell(f_i),i\in\{1,\ldots,M\})\in B$ and, for all $i$, $\left(\frac{\ell(e_j)}{\ell(f_i)},j\in S(i)\right)\in C_i$, then we have $\ell(e_i),i\in\{1,\ldots,K\})\in A.$
Then
\begin{align*}
\pr\Big[C_{\sigma}\equiv (G_1,\ldots,G_k), \ & (\ell(e_i),i\in\{1,\ldots,K\})\in A\Big] \\
           &\geq\E\left[\mathbf{1}_{\{\T_{\sigma}^{\mathrm{mk}}\equiv T,(\ell(f_i),i\in\{1,\ldots,M\})\in B\}}\prod_{j=1}^n\frac{\ell(f(y_j))}{\mathrm{len}(\T_{\sigma}^{\mathrm{mk}})}\prod_{i=1}^M \, \frac{1}{|S(i)|!}\mathbf{1}_{\{D(i)\in C_i\}}\right].
\end{align*}

By Lemma~\ref{lem:chargetree}, the event $\{\T_{\sigma}^{\mathrm{mk}}\equiv T,(\ell(f_i),i\in\{1,\ldots,M\})\in B\}$ occurs with positive probability, and since Dirichlet distributions charge the full simplex, we do indeed have that
\[\pr\Big[C_{\sigma}\equiv (G_1,\ldots,G_k),(\ell(e_i),i\in\{1,\ldots,K\})\in A\Big]>0.\]

We finally turn to the result for $\mathcal{C}_{\mathrm{cplx}}$. Recall that $(\sigma_i,i\geq 1)$ are the ranked excursion lengths of a Brownian motion with parabolic drift and that, conditionally on the lengths, $\mathcal{C}_i, i\geq 1$ are independent copies of $\mathcal{C}_{\sigma_i}.$ Notice that
\begin{align*}
& \pr[\mathcal{C}_{\mathrm{cplx}}\equiv(G_1,\ldots,G_k), \text{ lengths in }A] \\ 
& \qquad \geq \pr[\mathcal{C}_{1}\equiv(G_1,\ldots,G_k),\text{lengths in }A ,\, \mathcal{C}_i \text{ has no complex components } \forall i\geq 2].
\end{align*} From Propositions~\ref{prop:excursionlengths} and~\ref{prop:Poissonbounds}, we deduce that $(\mathcal{C}_i, i\geq 2)$ has no complex components with positive probability. An application of the first part of the proposition then completes the proof.
\end{proof}

\appendix

\section{Gromov--Hausdorff distances}
In this section, we give relevant background on the Gromov--Hausdorff distance and its variants that we use in this paper. For additional details and proofs, we refer to Chapter 7 of \cite{Burago}, Section 6 of \cite{MiermontTessellations}, and the references therein.

\subsection{Definitions}\label{sec:appendix1}
Consider two compact metric spaces $X$ and $X'$. The \emph{Gromov--Hausdorff distance} $d_{\mathrm{GH}}(X,X')$ between them is defined to be
\[
d_{\mathrm{GH}}(X,X')=\underset{\phi,\phi'}\inf d_{\mathrm{H},\mathcal{Z}} (\phi(X),\phi'(X)),
\]
where the infimum is taken over all possible isometric embeddings $\phi$ and $\phi'$ into a common metric space $\mathcal{Z},$ and $d_{\mathrm{H},\mathcal{Z}}$ denotes the Hausdorff distance between compact subsets of $\mathcal{Z}.$

We use two variants of the Gromov--Hausdorff distance which include \emph{marked points} and \emph{probability measures} on $X$ and $X'$, respectively. First, consider $k\in\N$ and points $x_i\in X$ and $x'_i\in X'$ for $i\in [k]$. We define the \emph{$k$-pointed Gromov--Hausdorff distance} between $(X,(x_i,i\in[k]))$ and $(X',(x'_i,i\in[k]))$ to be
\begin{equation} \label{eqn:kpointedGH}
d_{\mathrm{GH}}^{k}\left(\left(X,(x_i,i\in[k])\right),\left(X',(x'_i,i\in[k])\right)\right)=\underset{\phi,\phi'}\inf \left (d_{\mathrm{H},\mathcal{Z}} (\phi(X),\phi'(X))\vee \underset{i\in [k]}\max \,d_{\mathcal{Z}} (\phi(x_i),\phi'(x_i))\right),
\end{equation}
where $\phi,$ $\phi',$ and $\mathcal{Z}$ are as before and $d_{\mathcal{Z}}$ is the metric on $\mathcal{Z}.$  We will also need a version of the pointed Gromov--Hausdorff distance which allows for a random (but finite) number of marked points.  For compact metric spaces $X$ and $X'$ let $S \subset X$ and $S' \subset X'$ be such that $|S| < \infty$ and $|S'| < \infty$.  Then define
\begin{equation} \label{eqn:pointedGH}
d_{\mathrm{GH}}^{*}((X, S), (X',S')) = \begin{cases}
d_{\mathrm{GH}}^{k}((X,S), (X', S')) & \text{if $|S| = |S'| = k \ge 0$} \\
\infty & \text{otherwise.}
\end{cases}
\end{equation}
Next, let $\nu$ and $\nu'$ be Borel probability measures on $X$ and $X'$ respectively. The \emph{Gromov--Hausdorff--Prokhorov distance} between $(X,\nu)$ and $(X',\nu')$ is defined to be
\[
d_{\mathrm{GHP}}((X,\nu),(X',\nu'))=\underset{\phi,\phi'}\inf \left (d_{\mathrm{H},\mathcal{Z}} (\phi(X),\phi'(X))\vee d_{\mathrm{P},\mathcal{Z}} (\phi^{-1}(\nu),(\phi')^{-1}(\nu')\right),
\]
where $\phi,$ $\phi',$ and $\mathcal{Z}$ are as before and $d_{\mathrm{P},\mathcal{Z}}$ is the Prokhorov metric between probability measures on $\mathcal{Z}.$

Note that these definitions are flexible: we can add more probability measures, combine marks and measures, and so on. Write $d^{k,l}_{\mathrm{GHP}}$ for the $k$-pointed and $l$-measured Gromov--Hausdorff--Prokhorov distance.  We will make particular use of the distance $d^*_{\mathrm{GHP}}$ defined as follows: for $S \subset X$, $S' \subset X'$, natural numbers $L$ and $L'$, Borel probability measures $(\nu_i, i \in [L])$ on $X$ and $(\nu_i', i \in [L'])$ Borel probability measures on $X'$, let
\begin{align*}
& d^*_{\mathrm{GHP}}((X, S, (\nu_i, i \in [L])), (X', S', (\nu'_i, i \in [L'])))\\
&  = \begin{cases}
d^{k,l}_{\mathrm{GHP}}((X,S, (\nu_i, i \in [L])), (X', S', (\nu'_i, i \in [L']))) & \text{if $|S| = |S'| = k$ and $L = L' = l$,} \\
\infty & \text{otherwise}.
\end{cases}
\end{align*}
These distances all make their respective sets of isometry classes into Polish spaces. The following lemma is a variant of Proposition 10 in \cite{MiermontTessellations} and is particularly useful for us.
\begin{lem}\label{lem:GHPtopointed} Let $((X^n,(x^n_i,i\in[k]),(\nu^n_i,i\in[l])), n\in\N)$ be a sequence of random $k$-pointed and $l$-measured compact metric spaces which converges in distribution to $(X,(x_i,i\in[k]),(\nu_i,i\in[l]))$. For all $n\in\N$, conditionally on $(X^n,(x^n_i,i\in[k]),(\nu^n_i,i\in[l]))$, let $(y^n_i,i\in[l])$ be independent random variables taking values in $X^n$ with respective distributions $(\nu^n_i,i\in[l])$. Then the sequence of $(k+l)$-pointed metric spaces, $(X^n,(x^n_i,i\in[k]),(y^n_i,i\in[l]), n\in\N)$, converges in distribution as $n \to \infty$ to $(X,(x_i,i\in[k]),(y_i,i\in[l])),$ where the $(y_i,i\in[l])$ are defined analogously to the $(y^n_i, i \in [l])$.
\end{lem}

\subsection{Correspondences and their use}\label{sec:appendix2}
A \emph{correspondence} between $X$ and $X'$ is a relation $\mathcal{R}$ such that, for any $x\in X,$ there exists at least one $x'\in X'$ such that $x \mathcal{R} x'$ and, for any $x'\in X$, there exists at least one $x\in X'$ such that $x \mathcal{R} x'.$ Correspondences are a very convenient tool with which to study Gromov--Hausdorff distances, by quantifying them through their \emph{distortions}. The distortion of the correspondence $\mathcal{R}$ is defined by
\[\mathrm{dis}\,\mathcal{R}= \sup \left\{\left|d(x,y)-d(x',y')\right|: x\mathcal{R}x',y\mathcal{R}y'\right\}.\]
It is then classical that
\[
d_{\mathrm{GH}}(X,X')=\frac{1}{2}\underset{\mathcal{R}}\inf\mathrm{dis}\,\mathcal{R},
\]
where the infimum is taken over all correspondences $\mathcal{R}$ between $X$ and $X'$. 

The following useful lemma showcases the use of a correspondence to bound a pointed Gromov--Hausdorff distance.
\begin{lem}\label{lem:pointedGHexcursion}
Let $f$ and $g$ be two excursion functions on $[0,\sigma]$, and let $s_1,\ldots,s_k,s'_1,\ldots,s'_k$ be points in $[0,\sigma]$. Then
\[
d^k_{\mathrm{GH}}\left(\left(\T_f,p_f(s_1),\ldots,p_f(s_k)\right),\left(\T_f,p_g(s'_1),\ldots,p_g(s'_k)\right)\right) \leq 2 \|f-g\|+\omega_{\delta}(f),
\]
where $\delta=\sup_{1 \le i \le k} |s_i-s'_i|$ and $\omega_{\delta}(f)$ is the $\delta$-modulus of continuity of $f$.
\end{lem}
\begin{proof}
The relation $\mathcal{R}=\{p_f(s),p_g(s)\}\subset \T_f\times\T_g$ is well-known to be a correspondence, with $\mathrm{dis}(\mathcal{R})\leq 4 \|f-g\|.$ As in the proof of Theorem 7.3.25 in \cite{Burago}, we can then build a metric $d$ on the disjoint union $\T_f \cup \T_g$ which extends their intrinsic metrics by letting, for $s$ and $s'$ in $[0,\sigma],$
\[d(p_f(s),p_f(s'))=\underset{t\in[0,\sigma]}\inf\left\{d_f(s,t)+d_g(t,s')+\frac{1}{2}\mathrm{dis}(\mathcal{R})\right\}.\]
It is then straightforward to verify that, under this embedding, $d_{\mathrm{H}}(\T_f,\T_g)\le 2 \|f-g\|$ and \\$d(p_f(s_i),p_g(s'_i))\le 2\|f-g\|+\omega_{\delta}(f).$
\end{proof}

We end this section with a formulation of the multiply pointed and measured Gromov--Hausdorff--Prokhorov distance in terms of correspondences. Let $X$ and $X'$ be compact metric spaces with marked points $(x_i,i\in[k])$ and $(x'_i,i\in[k])$ as well as Borel probability measures $(\nu_i,i\in[l])$ and $(\nu'_i,i\in[l]).$ We let $R$ be the set of all correspondences $\mathcal{R}$ between $X$ and $X'$ such that $x_i \mathcal{R} x'_i$ for $i\in[k]$. For $i\in[l]$, let $\mathsf{C}(\nu_i,\nu_i')$ be the set of \emph{couplings} of $\nu_i$ and $\nu'_i$, namely Borel probability measures on $X\times X'$ which have $\nu_i$ and 
$\nu'_i$ as marginals on $X$ and $X'$ respectively. We then have
\begin{align*}d^k_{\mathrm{GHP}}&\left(\left(X,(x_i,i\in[k]),(\nu_i,i\in[l])\right),\left(X',(x'_i,i\in[k]),(\nu'_i,i\in[l])\right)\right)= \\ &\inf \left\{\rho>0: \exists\,  \mathcal{R} \in R(X,Y),  \mu_i \in \mathsf{C}(\nu_i,\nu_i'), i \in [l] \text{ such that } \underset{i\in[l]}\inf\mu_i(\mathcal{R})\ge 1-\rho, \mathrm{dis}\,\mathcal{R}\leq \rho \right\}. \numberthis \label{eq:GHPdistortion}
\end{align*}
\section*{Acknowledgements}

This research was supported by EPSRC Fellowship EP/N004833/1.  We would like to thank Nicolas Broutin and Julien Berestycki for helpful discussions.  We are very grateful to \'Eric Brunet for the proof of Proposition~\ref{prop:excursionlengths} $(ii)$. We would like to thank the referees for their careful reading of the paper and insightful comments, which led to many improvements.

\bibliographystyle{abbrv}
\bibliography{bib}

\end{document}